\documentclass[]{scrartcl}
\usepackage[T1]{fontenc}
\usepackage[utf8]{inputenc}
\usepackage{amsthm}
\usepackage{amssymb,amsmath}
\usepackage{mathtools}
\usepackage{enumitem}
\usepackage{etoolbox}
\usepackage{comment}

\usepackage[pdftex,colorlinks,linkcolor=black,urlcolor=black,citecolor=black]{hyperref}
\hypersetup{
	colorlinks=true,
    linkcolor={red!50!black},
    citecolor={blue!50!black},
    urlcolor={blue!80!black},
	bookmarksopen=true,
	bookmarksnumbered,
	bookmarksopenlevel=2,
	bookmarksdepth=3
}

\usepackage[%
    backend=biber,%
    style=ieee,%
    sortcites,%
    dashed=false,%
    citestyle=numeric,%
    sorting=nyvt,%
    giveninits=true,%
    url=false,%
    doi=true,%
    isbn=false,%
    mincitenames=1,%
    maxcitenames=3,%
    maxbibnames=99,%
    defernumbers=true,%
    ]{biblatex}
\usepackage[capitalise,nameinlink]{cleveref}
\newtheorem{theorem}{Theorem}[section]
\newtheorem{lemma}[theorem]{Lemma}
\newtheorem{proposition}[theorem]{Proposition}
\newtheorem{claim}{Claim}
\newtheorem{conjecture}[theorem]{Conjecture}
\newtheorem{remark}{Remark}
\newtheorem{corollary}[theorem]{Corollary}
\newtheorem{observation}[theorem]{Observation}

\Crefname{claim}{Claim}{Claims}
\Crefname{conjecture}{Conjecture}{Conjectures}
\Crefname{observation}{Observation}{Observations}
\Crefname{step}{Step}{Steps}
\Crefname{substep}{}{}

\usepackage{lineno}
\usepackage{tikz}
\usetikzlibrary{calc,decorations}
\tikzset{
  v:main/.style = {draw, circle, scale=0.8, thick,fill=black,inner sep=0.7mm},
  v:mainempty/.style = {draw, circle, scale=0.8, thick,fill=white,inner sep=0.7mm},
  v:middle/.style = {draw, circle, scale=0.3,thick,fill=Gray,color=Gray,inner sep=1mm},
  v:border/.style = {draw, circle, scale=0.75, thick,minimum size=10.5mm},
  v:mainfull/.style = {draw, circle, scale=1, thick,fill},
  v:ghost/.style = {inner sep=0pt,scale=1},
  v:empty/.style = {draw=none,fill=none},
  >={latex},
  e:shiftedright/.style = {decoration={sl, raise=0.65pt},  decorate},
  e:shiftedleft/.style  = {decoration={sl, raise=-0.65pt}, decorate},
  e:marker/.style = {line width=8.5pt,line cap=round,opacity=0.35,color=DarkGoldenrod},
  e:colored/.style = {line width=1.8pt,color=BostonUniversityRed,cap=round,opacity=0.8},
  e:coloredthin/.style = {line width=1.6pt,opacity=1},
  e:coloredborder/.style = {line width=3.4pt},
  e:main/.style = {line width=1pt},
  e:thick/.style = {line width=2pt},
  e:mainthin/.style = {line width=0.6pt},
}

\newenvironment{proofclaim}{
	\noindent \emph{Proof of claim.}
}{%
	\hfill $\diamond$ \\
}

\usepackage{authblk}

\usepackage[defaultlines=3,all]{nowidow}
\usepackage{microtype}

\newcommand{\tw}{\mathrm{tw}}
\newcommand{\tin}{\mathrm{tree}\text{-}\alpha}
\newcommand{\X}{\mathbb{X}}
\newcommand{\T}{\mathcal{T}}
\newcommand{\NP}{\textsf{NP}}

\newcommand\extrafootertext[1]{%
    \bgroup
    \renewcommand\thefootnote{\fnsymbol{footnote}}%
    \renewcommand\thempfootnote{\fnsymbol{mpfootnote}}%
    \footnotetext[0]{#1}%
    \egroup
}

\date{}
\title{Induced Minor Models. I. Structural Properties and Algorithmic Consequences\thanks{Parts of this work appeared in an extended abstract published in the proceedings of the  35th International Workshop on Combinatorial Algorithms (IWOCA 2024),  Springer, 2024, pp.~151--164.  %
  This work was funded by the French \emph{Fédération de Recherche ICVL} (Informatique Centre-Val de Loire),
  the Slovenian Research and Innovation Agency (I0-0035, research program P1-0285 and research projects J1-3001, J1-3002, J1-3003, J1-4008, J1-4084, and N1-0370),
  and the research program CogniCom (0013103) at the University of Primorska.
  The last author is supported by Projet ANR GODASse, Projet-ANR-24-CE48-4377}}

\author[1]{Nicolas Bousquet}
\author[2]{Clément Dallard}
\author[3]{Maël Dumas}
\author[4]{Claire Hilaire}
\author[4]{Martin Milani\v{c}}
\author[5]{Anthony Perez}
\author[6]{Nicolas Trotignon}

\newcommand{\email}[1]{
  \texttt{#1}
}

\affil[1]{CNRS, Université de Montréal CRM – CNRS, Canada\protect\\\email{nicolas.bousquet@cnrs.fr}}
\affil[2]{Department of Informatics, University of Fribourg, Switzerland\protect\\\email{clement.dallard@unifr.ch}}
\affil[3]{Institute of Informatics, University of Warsaw, Poland\protect\\\email{mdumas@mimuw.edu.pl}}
\affil[4]{FAMNIT and IAM, University of Primorska, Slovenia\protect\\\email{\{claire.hilaire,martin.milanic\}@upr.si}}
\affil[5]{Université d'Orléans, INSA CVL, LIFO UR 4022, France\protect\\\email{anthony.perez@univ-orleans.fr}}
\affil[6]{CNRS, ENS de Lyon, Université Claude Bernard Lyon 1, LIP, UMR 5668, France\protect\\\email{nicolas.trotignon@ens-lyon.fr}}

\makeatletter
\def\@maketitle{%
\newpage%
\null%
\begin{center}%
    \let\footnote\thanks %
    {\huge\bfseries \@title %
      \par
    }
  \vskip 1.5em
    {\lineskip .5em
     \begin{tabular}[t]{c}
        \baselineskip=12pt
        \@author
     \end{tabular}
     \par
    }
\end{center}
\par
\vskip 1.5em}
\makeatother

\begin{document}

    \maketitle

  \begin{abstract}
    A graph $H$ is said to be an \emph{induced minor} of a graph $G$ if $H$ can be obtained from $G$ by a sequence of vertex deletions and edge contractions.
    Equivalently, $H$ is an induced minor of $G$ if there exists an \emph{induced minor model} of $H$ in $G$, that is, a collection of pairwise disjoint subsets of vertices of $G$ labeled by the vertices of $H$, each inducing a connected subgraph in $G$, such that two vertices of $H$ are adjacent if and only if there is an edge in $G$ between the corresponding subsets.

    In this paper, we investigate structural properties of induced minor models, including bounds on treewidth and chromatic number of the subgraphs induced by minimal induced minor models.
    It is known that for some graphs $H$, testing whether a given graph $G$ contains $H$ as an induced minor is an \textsf{NP}-complete problem.
    Nevertheless, as algorithmic applications of our structural results, we make use of recent developments regarding tree-independence number to show that if $H$ is the $4$-wheel, the $5$-vertex complete graph minus an edge, or a complete bipartite graph $K_{2,q}$, then there is a polynomial-time algorithm to find in a given graph $G$ an induced minor model of $H$ in $G$, if there is one.
    We also develop an alternative polynomial-time algorithm for recognizing graphs that do not contain $K_{2,3}$ as an induced minor, which revolves around the idea of detecting the induced subgraphs whose presence is forced when the input graph contains $K_{2,3}$ as an induced minor, using the so-called shortest path detector.
    It turns out that all these induced subgraphs are Truemper configurations.
  \end{abstract}

\section{Introduction}\label{sec:intro}

One of the successful approaches to dealing with \textsf{NP}-hard graph problems is the identification of restrictions on input instances under which the problem becomes solvable in polynomial time.
Such restrictions are usually described in terms of \textsl{graph classes}, that is, families of graphs that satisfy some common property.
Many graph classes of practical and theoretical relevance are closed under various graph operations such as vertex deletions, edge deletions, and/or edge contractions.
In particular, classes closed under all these operations are \textsl{minor-closed}; they are well understood, both structurally and algorithmically: most notably, any such graph class can be characterized by a finite \textsl{obstruction set} (that is, the family of minimal forbidden minors, see~\cite{MR2099147}) and recognized in polynomial time (see~\cite{MR1309358}).
In the much more general situation when the class is only closed under vertex deletions, we obtain the family of \textsl{hereditary} graph classes; in this case, the recognition problem is always polynomial if the class is characterized by a finite obstruction set, but can be \textsf{NP}-complete in general (as, for example, in the case of $3$-colorable graphs).

In this paper, we focus on an intermediate variant: graph classes that are closed under vertex deletions and edge contractions; in other words, they are closed under \textsl{induced minors}.
Many graph classes studied in the literature enjoy this property, including chordal graphs~\cite{MR1320296}, circular-arc graphs, their common generalizations $1$-perfectly orientable graphs (see, e.g.,~\cite{MR3647815,MR3152051}) and bisimplicial graphs~\cite{milanic2023bisimplicial}, polygon-circle graphs~\cite{MR1428585}, several geometrically defined graph classes such as the class of planar graphs, and various intersection graph classes.
More precisely, for any graph class $\mathcal G$, the class of graphs defined as intersection graphs of connected subgraphs of graphs in $\mathcal G$ is closed under induced minors.

A recent trend in structural and algorithmic graph theory has been to try to generalize the theory of graph minors to the setting of induced minors.
While this is a much more challenging setting (see, e.g.,~\cite{aboulker2025induceddisjointpathsinduced,bonnet2025inducedminorsregionintersection}), the world of induced-minor-closed classes enjoys several good structural and algorithmic properties (see, e.g.,~\cite{MR4839678,korhonen2023inducedminorfree,DBLP:conf/soda/AhnGHK25,KORHONEN2023206,DBLP:conf/esa/BonnetDGTW23,hickingbotham2025inducedminorsasymptoticdimension}).

Every induced-minor-closed class can be characterized by a family $\mathcal F$ of forbidden induced minors.
The set $\mathcal F$ may be finite (for instance, $\mathcal{F} = \{K_{2,3}\}$ yields the class of graphs in which each minimal separator induces a subgraph with independence number at most $2$; see, e.g.,~\cite{DMS_JCTB2024,chakraborty2025k23inducedminorfreegraphsadmit}) or infinite (as in the case of bisimplicial graphs, which are defined as graphs in which each minimal separator is a union of two cliques, see~\cite{milanic2023bisimplicial}).\footnote{For a simpler example, consider the case when $\mathcal F$ consists of all complements of cycles.}
However, contrary to the case of the induced subgraph relation, finiteness of the family of forbidden induced minors does not imply polynomial-time recognition (unless \textsf{P} = \textsf{NP}).
Indeed, as shown by Fellows, Kratochv\'{\i}l, Middendorf, and Pfeiffer in 1995~\cite{MR1308575}, there exists a graph $H$ such that determining if a given graph $G$ contains $H$ as an induced minor is \textsf{NP}-complete.
Recently, Korhonen and Lokshtanov showed that this can happen even when $H$ is a tree~\cite{korhonen2023inducedminorfree}. Even more recently, Aboulker, Bonnet, Picavet, and Trotignon~\cite{aboulker2025induceddisjointpathsinduced} showed that this can also happen when $H$ is subcubic and without two adjacent degree-$3$ vertices.

It is thus a natural question to determine for which graphs $H$ deciding if a given graph $G$ contains $H$ as an induced minor can be done in polynomial time.
This is the case for all graphs $H$ on at most four vertices.
The problem is particularly simple if $H$ is a forest on at most four vertices or the disjoint union of a vertex and a triangle.
Otherwise, excluding $H$ always leads to a well structured class that is easy to recognize.
When $H$ is the $4$-cycle, that is, $K_{2,2}$, the class of graphs excluding $H$ as an induced minor is precisely the class of chordal graphs.
For $H = K_4$, excluding $H$ as an induced minor is equivalent to excluding it as a minor, which leads to the class of graphs with treewidth at most two.
For the diamond, that is, the graph $H$ obtained from $K_4$ by removing one edge, excluding $H$ leads to graphs whose blocks are all complete graphs or chordless cycles (see~\cite[Theorem 4.6]{hartinger2017new}).
For the paw, that is, the graph $H$ obtained from $K_4$ by removing two adjacent edges, a direct algorithm is easy.
In this latter case, we note that excluding $H$ as an induced minor is equivalent to excluding $H$ as an induced topological minor, and this leads to a very simple class~\cite{DBLP:journals/jgt/ChudnovskyPST12}.

For graphs with more than four vertices, Fiala, Kamiński, and Paulusma~\cite{FKM12} showed that the problem of deciding whether a fixed graph $H$ is an induced minor of a given graph is polynomial if $H$ is a forest with at most seven vertices, except for one remaining open case.
In the same paper, the authors also showed that it is polynomial whenever $H$ is a subdivided star or a star to which two leaves are added to a vertex of degree one.
More recently, three infinite families of graphs have been identified, for which the problem becomes polynomial whenever $H$ belongs to one of them~\cite{DDHP2024}.
These three families are the \emph{flowers} (informally, paths, cycles, and diamonds intersecting in one vertex), stars whose center can have up to two true twins, and the generalized houses and bulls (paths or cycles plus one vertex adjacent to both endpoint of an edge).

If we restrict the class of graphs where we want to decide the presence of an induced minor, then the problem can be solved in linear time for every fixed graph $H$ when working with graphs of bounded treewidth, using the celebrated metatheorem of Courcelle~\cite{MR1042649}.
Furthermore, since the property that a graph $G$ admits an induced minor model of $H$ can be expressed as an $\mathsf{MSO}_1$ formula (see \cref{sec:imt-tree-independence}),
it is even possible to solve the problem in polynomial time in graphs of bounded cliquewidth using metatheorem due to Courcelle, Makowsky, and Rotics~\cite{MR1739644}.
For incomparable classes of graphs, several ad-hoc results have been obtained, for instance for planar graphs~\cite{MR1308575}, or, more generally, for proper minor-closed classes\footnote{A graph class $\mathcal{G}$ is \emph{proper} minor-closed if it is minor-closed and is not the class of all graphs.} if $H$ is planar~\cite{HKPST12}, AT-free graphs~\cite{GKP13}, and chordal graphs~\cite{BGHHKP11}.

\subsubsection*{Results and organization of the paper}

We first prove in \cref{sec:deciding_finding} that in any induced-minor-closed class $\mathcal{G}$, the problem of finding an induced minor model of a fixed graph $H$ in a graph $G \in \mathcal{G}$, if one exists, can be polynomially reduced to the problem of deciding if $H$ is an induced minor of $G$.
We then investigate, in \Cref{parameters of models}, structural properties of induced minor models.
More precisely, we show in \Cref{subsec:tw-branch-sets} that in any minimal induced minor model of a graph $H$ in a graph $G$, the treewidth of the subgraph of $G$ induced by any branch set is bounded from above by a function of $H$ alone.
This result motivates the question of whether the treewidth of the subgraph induced by the entire induced minor model can also be bounded.
We investigate this question in \Cref{subsec:tw-model}, where we show that this statement fails in general, even if $H$ is a tree and when treewidth is replaced by degeneracy.
We prove, however, that the result holds if every edge of $H$ is incident with a vertex of degree at most $2$.
Then, in \Cref{subsec:omega-model}, we show that a relaxation of the above property holds, namely that for any minimal induced minor model of a graph $H$ in a graph $G$, the chromatic number of the subgraph of $G$ induced by the entire induced minor model is bounded from above by a function of $H$ alone.

We complement these results by showing in \cref{sec:imt-tree-independence} that deciding if a given graph $G$ with bounded tree-independence number contains a fixed graph $H$ as an induced minor can be done in polynomial time.
We refer the reader to \cref{sec:preliminaries} for a formal definition of tree-independence number.
As a byproduct, we obtain polynomial-time algorithms to detect the $4$-wheel, the $5$-vertex complete graph minus an edge and the complete bipartite graph $K_{2,q}$ as an induced minor, for any fixed nonnegative integer $q$. Together with the results presented in this paper, the aforementioned results show that the problem of finding $H$ as an induced minor is polynomial-time solvable for all but three graphs on five vertices: $K_4$ with a pendant vertex, the same graph minus an edge incident to the neighbor of the pendant vertex, and a $K_4$ with a subdivided edge.
We note moreover that, by a result of  Dallard, Milani\v{c}, and \v{S}torgel~\cite[Lemma 3.2]{DMS_JCTB2024}, graphs not containing $K_{2,q}$ as an induced minor are exactly the graphs in which all minimal separators have independence number less than $q$.
In particular, graphs that do not contain $K_{2,3}$ as an induced minor have minimal separators with independence number at most $2$, and are hence a natural generalization of chordal graphs (where all minimal separators are cliques).
We focus on this particular case in \cref{sec:k23} and provide an alternative polynomial-time algorithm to detect $K_{2,3}$ as an induced minor.\footnote{This was the main result presented in the extended abstract published at IWOCA 2024~\cite{DBLP:conf/iwoca/DallardDHMPT24}.}
Our algorithm uses the so-called shortest path detector method to detect the induced subgraphs whose presence is forced when the input graph contains $K_{2,3}$ as an induced minor.
It turns out that all these induced subgraphs are Truemper configurations.

Observe that detecting $K_{2, 3}$ as an induced minor might look easy, in particular because the topological version (that is, detecting a theta) is known already.
We might indeed have overlooked a simpler method to solve it.
However, several \textsf{NP}-completeness results for questions similar to the ones that we address here (namely the detection of prisms and broken wheel, see below for the details) suggest that a too naive attempt is likely to fail.

\section{Preliminaries}
\label{sec:preliminaries}

\subsection{Definitions and notations}

In this article, we consider finite, simple, and undirected graphs with $n$ vertices and $m$ edges.
Given a graph $G$ and a set $S\subseteq V(G)$, the \emph{subgraph of $G$ induced by $S$} is the graph denoted by $G[S]$ with vertex set $S$ in which two vertices are adjacent if and only if they are adjacent in $G$.
We extend this definition so that, for a set $\X$ of subsets of vertices, the \emph{subgraph of $G$ induced by $\X$} is the graph $G[\bigcup_{X \in \X} X]$, denoted by $G[\X]$.
    We denote by $G\setminus S$ the subgraph of $G$ induced by $V(G)\setminus S$, and when $S=\{v\}$, we simply write $G\setminus v$.
  When there is no ambiguity, we do not distinguish between an induced subgraph and its vertex set.

A class of graphs is \emph{hereditary} if it is closed under taking induced subgraphs.
A \emph{subdivision} of a graph $G$ is any graph obtained from $G$ by a sequence of edge subdivisions.
Given two graphs $G$ and $H$, the graph $H$ is an \emph{induced subgraph} of $G$ if $H$ can be obtained from $G$ by a sequence of vertex deletions, an \emph{induced minor} of $G$ if $H$ can be obtained from $G$ by a sequence of vertex deletions or edge contractions, and an \emph{induced topological minor} of $G$ if some subdivision of $H$ is an induced subgraph of $G$.

Given two graphs $G$ and $H$, we say that we \emph{detect} $H$ as an induced subgraph (respectively as an induced minor) of $G$ if we decide whether $G$ contains $H$ as an induced subgraph (respectively as an induced minor).
If the answer is no, then we say that $G$ is \emph{$H$-free} (respectively $G$ is \emph{$H$-induced-minor-free}).
When $H$ is replaced with a (possibly infinite) family $\mathcal{H}$ of graphs, the above definitions are extended in a natural way; for instance, \emph{detecting} $\mathcal{H}$ means deciding whether some graph of $\mathcal{H}$ is an induced subgraph of $G$ (respectively an induced minor of $G$).

Given a graph $G$ and a set $S\subseteq V(G)$, we denote by $N_G(S)$ the set of vertices in $V(G)\setminus S$ having a neighbor in $S$ and by $N_G[S]$ the set $S\cup N_G(S)$.
For $S = \{v\}$ for some $v\in V(G)$, we write $N_G(v)$ for $N_G(\{v\})$ and similarly for $N_G[v]$.
When $G$ is clear from context, the subscript $G$ will be omitted.
Given a graph $G$ and a vertex $v$ of $G$, we denote by $d_G(v)$ the degree of $v$ in $G$, and $\Delta(G)$ the maximum degree of $G$.

Two disjoint subgraphs, or vertex sets, in a graph $G$ are \emph{complete} to each other if every two vertices in different subgraphs, or vertex sets, are adjacent.
Similarly, they are \emph{anticomplete} to each other if no two vertices in different subgraphs, or vertex sets, are adjacent.

A \emph{path} in a graph $G$ is a sequence $v_1 \ldots v_k$ of pairwise distinct vertices of $G$ such that $v_i$ is adjacent to $v_{i+1}$, for every integer $i$ such that $0\leq i \leq k-1$.
We also view the empty set as an empty path.
Given a nonempty path $P = v_1\ldots v_k$, we define the \emph{length} of $P$ as $k-1$, its \emph{endpoints} as the vertices $v_1$ and $v_k$, and its \emph{interior} as the subpath $v_2\ldots v_{k-1}$ (which is empty if $k\le 2$). We also call \emph{internal vertices} the vertices in the interior of $P$.
A \emph{chord} of a path in $G$ is an edge in $G$ whose endpoints are non-consecutive vertices of the path.
A path is \emph{chordless} if it has no chord.
Given a path $P = v_1\ldots v_k$ and two vertices $a$ and $b$ on $P$ such that $a = v_i$ and $b = v_j$ where $1\le i\le j\le k$, we denote by $aPb$ the subpath of $P$ from $a$ to $b$, that is, the path $v_i\ldots v_j$.
Given an integer $k\ge 3$, a \emph{cycle of length $k$} in a graph $G$ is a sequence $v_1\ldots v_kv_1$ of vertices of $G$ such that the vertices $v_1,\ldots, v_k$ are pairwise distinct, $v_i$ is adjacent to $v_{i+1}$ for all $i\in \{1,\ldots, k-1\}$, and $v_k$ is adjacent to $v_1$.
A \emph{chord} of a cycle in $G$ is an edge in $G$ whose endpoints are non-consecutive vertices of the cycle.
A cycle is \emph{chordless} if it has no chord.
For convenience, and when there is no ambiguity, we do not distinguish between a path and its vertex set, and the same for a cycle.

An \emph{independent set} in a graph $G$ is a set of pairwise nonadjacent vertices.
A \emph{clique} in $G$ is a set of pairwise adjacent vertices.
The \emph{clique number} of a graph $G$, denoted by $\omega(G)$, is the maximum cardinality of a clique in $G$.
A \emph{triangle} is a clique of size~$3$.
The graph $K_{p,q}$ is a complete bipartite graph with parts of size $p$ and $q$, respectively.
The \emph{chromatic number} of a graph $G$, denoted by $\chi(G)$, is the minimum integer $k$ such that the vertex set of $G$ is a union of $k$ independent sets.

Given a graph $G$, a \emph{tree decomposition} of $G$ is a pair $\mathcal{T}=(T,\beta)$ consisting of a tree $T$ and a function $\beta\colon V(T)\to 2^{V(G)}$ whose images are called the \emph{bags} of $\mathcal{T}$ such that every vertex of $G$ belongs to some bag, for every $e\in E(G)$ there exists some $t\in V(T)$ with $e\subseteq \beta(t)$, and for every vertex $v\in V(G)$ the set $\{t\in V(T) \colon v\in \beta(t)\}$ induces a subtree of $T$.
We refer to the vertices of $T$ as the \emph{nodes} of the tree decomposition $\mathcal{T}$.
If $T$ is a path, then we call $\mathcal{T}$ a \emph{path decomposition} of $G$.
The \emph{width} of $\mathcal{T}$ equals $\max_{t \in V(T)} |\beta(t)|-1$, and the \emph{treewidth} of a graph $G$, denoted by $\tw(G)$, is the minimum possible width of a tree decomposition of $G$.
The \emph{independence number} of $\mathcal{T}$, denoted by $\alpha(\mathcal{T})$, is defined as  \[\alpha(\mathcal{T})=\max_{t\in V(T)} \alpha(G[\beta(t)]).\]
The \emph{tree-independence number} of a graph $G$, denoted by $\tin(G)$, is the minimum independence number among all possible tree decompositions of $G$.
Observe that every graph $G$ satisfies $\tin(G) \leq \alpha(G)$.

\subsection{Minimal induced minor models}

Let $G$ be a graph and $H$ be a graph that is an induced minor of $G$.
An \emph{induced minor model} of $H$ %
is a collection $\X = \{X_u \colon u \in V(H)\}$ of subsets of $V(G)$ indexed by vertices of $H$ such that:
\begin{itemize}
  \item for $u \neq v \in V(H)$, $X_u \cap X_v = \emptyset$,
  \item for $u \in V(H)$, the graph $G[X_u]$ is connected, and
  \item for $u \neq v \in V(H)$, there exists an edge of $G$ between $X_u$ and $X_v$ if and only if $uv \in E(H)$.
\end{itemize}
Each set $X_u \in \X$ is called a \emph{branch set} of $\X$.
We define the treewidth of \emph{$\X$} as the treewidth of $G[\X]$, the subgraph of $G$ induced by $\X$ (and similarly for degeneracy).

An induced minor model $\X = \{X_u \colon u \in V(H)\}$ of $H$ in $G$ is said to be \emph{minimal} if $\X$ induces in $G$ a subgraph $G'$ such that no proper induced subgraph of $G'$ contains $H$ as an induced minor.

The following lemma can be derived from the proof of \cite[Proposition 3.2]{milanic2024tamevsferaldichotomy}.
To keep the paper self-contained, we include a proof using our notation and terminology.

\begin{lemma}\label{lem: min model deg 2}
  Let $H$ and $G$ be two graphs such that $G$ contains $H$ as an induced minor.
  Then, there exists a minimal induced minor model $\X = \{X_u \colon u \in V(H)\}$ of $H$ in $G$ with the following properties.
  \begin{enumerate}
    \item $|X_u| = 1$ for every vertex $u\in V(H)$ that has degree at most $2$ and belongs to a connected component of $H$ that is not a cycle.
    \item For every connected component $C$ of $H$  that is a cycle, there exists a vertex $v$ of $C$ such that $G[X_v]$ is a path and $|X_u| = 1$ for every vertex $u$ in $V(C-v)$.
  \end{enumerate}
\end{lemma}
\begin{proof}

  Let $\X' = \{X_u' \colon u \in V(H)\}$ be a minimal induced minor model of $H$ in $G$.
  Let $u$ be a vertex of degree at most $2$ in $H$.
  If $d_H(u)\leq 1$, then $|X_u'|=1$, by the minimality of $\X'$.

  Moreover, if $d_H(u)= 2$, with neighbors $v$ and $w$, then there is a unique vertex $x_v$ in $N(X_v')\cap X_u'$, and a unique vertex $x_w$ in $N(X_w')\cap X_u'$, the set $X_u'$ induces in $G$ a path $Q$ whose endpoints are $x_v$ and $x_w$, and no internal vertex of $Q$ is adjacent to any vertex in a branch set other than $X_u'$.
  This property extends to paths of vertices of degree $2$ as follows.

  \begin{observation}\label{cl: deg 2}
    Let $P$ be a path in $H$ consisting only of vertices of degree $2$ in $H$, let $v$ and $w$ be the neighbors not in $P$ of the endpoints of $P$ (possibly $v=w$), and let $X'_P$ be the union of the branch sets of the vertices of $P$.
    Then, there is a unique vertex $x_v$ in $N(X_v')\cap X'_P$ and a unique vertex $x_w$ in $N(X_w')\cap X_P'$, and $X_P'$ induces in $G$ a path $Q$ whose endpoints are $x_v$ and $x_w$, and no internal vertex of $Q$ is adjacent to a vertex in a branch set not included in $X_P'$.
  \end{observation}

  We now construct an induced minor model $\X= \{X_u \colon u \in V(H)\}$ of $H$ in the subgraph $G' = G[\X']$ that satisfies the properties on the vertices of degree $2$ in $H$ stated in the lemma.
  This suffices to prove the lemma, since the definition of $G'$ and the minimality of $\X'$ imply that any such induced minor model $\X$ of $H$ in $G$ is minimal, and thus necessarily satisfies the desired property on the vertices of degree at most $1$.
  Note that if $H$ is not connected, then $G'$ is not connected and we can construct $\X$ independently for each connected component of $H$ in the corresponding connected component of $G'$.
  Hence, we can assume that $H$ is connected.

  Suppose first that $H$ is a cycle, say $u_1u_2\dots u_\ell$. Then by \cref{cl: deg 2}, we get that $G'$ is as cycle of length $\ell'\geq \ell$, say $w_1w_2\dots w_{\ell'}$.
  Then it suffices to set $X_{u_i}=\{w_i\}$ for each $i<\ell$, and $X_{u_\ell}$ contains the remaining vertices.

  Suppose now that $H$ is not a cycle.
  Let $P=u_1\dots u_\ell$ be a maximal path in $H$ such that every vertex of $P$ has degree~$2$ in $H$.
  Since $u_1$ and $u_{\ell}$ have degree $2$ in $H$ and $H$ is not a cycle, there exist two (not necessarily distinct) vertices $u_0$ and $u_{\ell+1}$ in $H$ such that $u_0$ is adjacent to $u_1$, and $u_{\ell+1}$ is adjacent to $u_\ell$, and $N_H(V(P)) = \{u_0,u_{\ell+1}\}$.
  By \cref{cl: deg 2}, the union of $X_{u_i}'$ for $1\leq i\leq \ell$ induces a path $P'=w_1\dots w_{q}$ in $G'$ with $q \geq \ell$. Moreover, $w_1$ has a neighbor in $ X'_{u_0}$, and $w_q$ has a neighbor in $ X'_{u_{\ell+1}}$, and all the other vertices of $P'$ have degree $2$ in $G'$.
  Let $Q'$ be the path obtained by removing $w_1,\dots,w_\ell$ from $P'$ (possibly $Q'$ is empty).
  Then, either $Q'$ is nonempty, in which case $N(V(Q'))$ contains $w_\ell$ as well as a vertex in $ X'_{u_{\ell+1}}$, or $Q'$ is empty, in which case $w_\ell$ has a neighbor in $ X'_{u_{\ell+1}}$.

  Thus, we can set $X_{u_i}=\{w_i\}$ for each $1\leq i\leq \ell$, $X_{u_0}=X'_{u_0}$, and $ X_{u_{\ell+1}} = V(Q')\cup  X'_{u_{\ell+1}}$, and $X_u=X'_u$ for every vertex $u$ in $V(H)\setminus V(P)$.
  We then get a new induced minor model of $H$ in $G'$ such that each vertex of $P$ has a branch set of size $1$.
  Moreover, the new branch sets of the neighbors of $P$ (which do not have degree $2$) contain their previous branch sets, and all the other branch sets are the same.
  We can thus repeat this construction for each path of $H$ with internal vertices of degree $2$.
  Since every vertex of degree $2$ belongs to a unique maximal path consisting only of vertices of degree $2$ in $H$, we get in the end an induced minor model with the desired property.
 Recall that we constructed this model $\X$ in $G'$, the subgraph of $G$ induced by $\X'$ which is minimal, thus $\X$ is necessarily a minimal induced model in $G$. If in this process we got a branch set of size more than one for a vertex of degree at most $1$, this would contradict the minimality of $\X$ and thus of $\X'$, hence this does not happen and $\X$ satisfies the properties in  the statement.
\end{proof}

\section{Deciding is finding}\label{sec:deciding_finding}

  When determining whether a graph $G$ contains a graph $H$ as an induced minor, we might be interested in finding the model of $H$ in $G$, but it could happen that the witness (in the positive case) is something else.
  In this section, we show that any polynomial-time algorithm detecting $H$ as an induced minor in a graph class $\mathcal{G}$ can be transformed into a polynomial-time algorithm that also determines an induced minor model of $H$ in $G$ in the positive case, provided that $\mathcal{G}$ is closed under taking induced minors.
  In fact, we show this for any algorithm deciding the presence of some member of a (possibly infinite) family $\mathcal{H}$ of graphs, instead of a single graph $H$.

  We start by observing that this transformation can be obtained easily if we replace the induced minor relation by the induced subgraph relation.

  \begin{observation}\label{obs sub algo}
    Let $\mathcal{H}$ be a family of graphs and $\mathcal{G}$ a hereditary graph class such that there exists a polynomial-time algorithm to determine whether a given graph $G \in \mathcal{G}$ contains a member of $\mathcal{H}$ as an induced subgraph.
    Then there exists a polynomial-time algorithm that, given a graph $G \in \mathcal{G}$, either returns a minimal induced subgraph of $G$ that belongs to $\mathcal{H}$ or determines that $G$ is $\mathcal{H}$-free.
  \end{observation}

  \begin{proof}
    Let us denote by $\mathcal{A}$ a polynomial-time algorithm that decides if an input graph $G$ contains a member of $\mathcal{H}$ as an induced subgraph.
    If $\mathcal{A}$ answers \textsc{No} when given $G$, then we return \textsc{No}.
    Let $G' \coloneq G$ be a copy of $G$.
    We iterate over the vertex set of $G'$ and, for every vertex $v \in V(G')$, we call $\mathcal{A}$ on $G'-v$.
    If the answer is \textsc{Yes}, then we set $G' \coloneq G'-v$; otherwise, we do not change $G'$.
    At the end of the iteration, $G'$ is a minimal induced subgraph of $G$ containing a member of $\mathcal{H}$ as an induced subgraph, in particular, $G'$ belongs to $\mathcal{H}$, but no proper induced subgraph of $G'$ belongs to $\mathcal{H}$.
  \end{proof}

  We show that this observation can be generalized for induced minor testing.

  \begin{lemma}\label{lem algo}
    Let $\mathcal{H}$ be a family of graphs and $\mathcal{G}$ a graph class closed under taking induced minors such that there exists a polynomial-time algorithm to detect whether a given graph $G \in \mathcal{G}$ contains a member of $\mathcal{H}$ as an induced minor.
    Then there exists a polynomial-time algorithm that, given a graph $G \in \mathcal{G}$, either returns a minimal induced minor model of some $H\in \mathcal{H}$ in $G$  or determines that $G$ is $\mathcal{H}$-induced-minor-free.
  \end{lemma}
  \begin{proof}

    Since $\mathcal{G}$ is closed under taking induced minors, $\mathcal{G}$ is hereditary.
    Let $\mathcal{A}$ be a polynomial-time algorithm that decides if an input graph $G$ contains a graph in $\mathcal{H}$ as an induced minor.
    Let $\mathcal{H}'$ be the family of graphs minimal under taking induced subgraphs that contain a member of $\mathcal{H}$ as an induced minor.
      Observe that a graph $G$ is $\mathcal{H}$-induced-minor-free if and only if $G$ is $\mathcal{H}'$-free.
        Hence,
    $\mathcal{A}$ is a polynomial-time algorithm that decides if an input graph $G$ contains a graph in $\mathcal{H}'$ as an induced subgraph.
        By \cref{obs sub algo}, there exists a polynomial-time algorithm $\mathcal{A}'$ that, given a graph $G \in \mathcal{G}$, either returns a minimal induced subgraph $G_\mathcal{H}$ of $G$ that belongs to $\mathcal{H}'$ or determines that $G$ is $\mathcal{H}'$-free.
        In the former case, $G_\mathcal{H}$ is a subgraph of $G$ induced by a minimal induced minor model of a graph $H \in \mathcal{H}$ in $G$, while in the latter case, $G$ is \hbox{$\mathcal{H}$-induced-minor-free}.
        To obtain the algorithm from the claim, we also need to determine, in the former case, the partition of the vertices of $G_\mathcal{H}$ into branch sets of a minimal induced minor model in $G$ of some graph in $\mathcal{H}$.

    Let $G$ be a graph from $\mathcal{G}$ and apply $\mathcal{A}'$ on $G$.
    If $\mathcal{A}'$ answers \textsc{No} when given $G$, then $G$ is $\mathcal{H}$-induced-minor-free, and we return \textsc{No}.
    So we can assume that $G$ contains a graph in $\mathcal{H}$ as an induced minor, and let $G_\mathcal{H}$ be the subgraph induced by a minimal induced minor model of some graph from $\mathcal{H}$ in $G$ returned by $\mathcal{A}'$.
    Note that every induced minor model of a graph $H\in\mathcal{H}$ in $G_\mathcal{H}$ is a minimal model of $H$ in $G$, and thus the goal is to correctly partition the vertices of $G_\mathcal{H}$ into branch sets.
    For that, we first identify a set of edges $E_{c}$ whose contraction in $G_\mathcal{H}$ results in a graph of $\mathcal{H}$.

    Let us start with $G'=G_\mathcal{H}$, $E_{c}=E_{\bar{c}}=\emptyset$, and maintain the fact that $G'$ contains a graph of $\mathcal{H}$ as an induced minor and is obtained from $G_\mathcal{H}$ by contracting the edges in $E_{c}$. Moreover, we maintain a set $E_{\bar{c}}$ such that if we contract the edges in $E_{c}$ and an edge in $E_{\bar{c}}$, then the resulting graph does not contain any graph of $\mathcal{H}$ as induced minor.

    For every edge $f$ of $G'$, let us denote $E_f$ the set of edges in $G_\mathcal{H}$ corresponding to $f$ in $G'$, \textit{i.e.}, the edges in $G_\mathcal{H}$ having endpoints in connected subgraphs $A$ and $B$ of $G_\mathcal{H}$ such that contracting the edges in $E_{c}$ maps $A$ and $B$ to the endpoints of $f$, respectively.
    Notice that contracting $f$ in $G'$ results in the same graph as contracting all the edges in $E_{c}$ plus an edge in $E_f$.
    In particular, if $E_f\cap E_{\bar{c}}\neq \emptyset$, then contracting $f$ in $G'$ results in a graph that does not admit a graph of $\mathcal{H}$ as induced minor.

    While $G'\notin\mathcal{H}$, there has to be some edge $f$ whose contraction results in a graph that contains a graph of $\mathcal{H}$ as induced minor, thus $E_f\cap E_{\bar{c}}=\emptyset$.
    Let us pick such an edge $f$.
    We apply $\mathcal{A}$ on the graph $G'_f$ obtained by contracting $f$ in $G'$.
    \begin{itemize}
      \item If the answer is \textsc{Yes}, we update $G':= G'_f$ and $E_{c}:= E_{c}\cup \{e\}$ where $e$ is an arbitrary edge in $E_f$.
      \item If the answer is \textsc{No}, then in every induced minor model obtained by contracting the edges in $E_{c}$ (we know there exists at least one such model by definition of $E_{c}$), we cannot obtain a graph in $\mathcal{H}$
            by contracting also an edge of $E_f$.
            Thus, we keep the same graph $G'$ for the next step and update $E_{\bar{c}}:= E_{\bar{c}}\cup E_f$.
    \end{itemize}

    At the end, $G'$ is isomorphic to a graph $H\in\mathcal{H}$ and obtained from $G_\mathcal{H}$ by contracting all the edges in $E_{c}$, and this algorithm requires at most $|E(G)|\leq |V(G)|^2$ calls to $\mathcal{A}$.

    Now, since $H$ is obtained by contracting $E_{c}$, observe that the graph $G''$ with vertex set $V(G_\mathcal{H})$ and edge set $E_{c}$ consists of $|V(H)|$ connected components, each corresponding to a branch set.
    Hence, to determine a corresponding induced minor model, it suffices to iterate over all the $|V(H)|!$ assignments of the connected components of $G''$ as branch sets of the induced minor model and, as soon as one assignment yields an induced minor model of $H$, return the corresponding induced minor model.
  \end{proof}

\section{Treewidth and chromatic number of induced minor models}\label{parameters of models}

In this section, we investigate structural properties of minimal induced minor models of a graph $H$ in a graph $G$, more precisely, what can be said about the parameters of the subgraphs of $G$ induced either by the individual branch sets or by their union.
The parameters we consider are treewidth in \Cref{subsec:tw-branch-sets,subsec:tw-model} and chromatic number in \Cref{subsec:omega-model}.

\subsection{Minimal models have branch sets with small treewidth}\label{subsec:tw-branch-sets}

We start with a lemma about tree decompositions of minimal induced subgraphs connecting a given set of vertices.

\begin{lemma}\label{small tw}
  Let $G$ be a graph and $S\subseteq V(G)$.
  Let $X\subseteq V(G)$ be a set such that $S \subseteq X$, $G[X]$ is connected, and $X$ is minimal with respect to these properties.
  Then, there is a tree decomposition $\T$ of $G[X]$ of width at most $|S|$ such that one of the bags contains $S$ and the tree decomposition obtained from $\T$ by removing vertices in $S$ from the bags is a tree decomposition of $G[X]-S$ of width at most $|S|-1$.
\end{lemma}

\begin{proof}
  We use induction on $|S|$.
  The statement is clear for $|S| \leq 1$.
  Assume now that $|S|\ge 2$.
  By the minimality of $X$, all the leaves of a tree subgraph of $G[X]$ that spans $S$ are in $S$.
  Fix such a leaf $v$ and note that the graph $G[X\setminus \{v\}]$ is connected.
  Consider a minimal subset $X'$ of $X$ containing $S'=S\setminus \{v\}$ such that $G[X']$ is connected.

  By the induction hypothesis, there is a tree decomposition $\mathcal{T'}=(T',\beta')$ of $G[X']$ width at most $|S|-1$, with $t'\in V(T')$ such that $S'\subseteq \beta(t')$. Moreover, removing $S'$ from each bag of $\mathcal{T'}$ results in a tree decomposition of $G[X']-S'$ of width at most $|S|-2$, that is, $|\beta'(t)\setminus S'|\le |S|-1$ for all $t\in V(T')$.

  By the minimality of $X$, the set $X$ is the disjoint union of $X'$ and the vertex set of a shortest path $P= v \dots u$ in $G[X]$ from $v$ to a neighbor of $X'$.
  By the minimality of $P$, the path $P$ is induced and $V(P)\setminus \{u\}$ is anticomplete to $X'$ (note that possibly $u=v$).

  Let $\mathcal{R}=(R,\beta_R)$ be a tree decomposition of width $1$ of $P$.
  We construct a tree decomposition $\mathcal T =(T,\beta)$ of $G[X]$ such that $T$ is the tree obtained by adding an edge between some node of $R$ and the node $t'$ of $T'$, and for each $t\in V(T)$, we have $\beta(t)=\beta'(t)\cup \{u\}$ if $t\in V(T')$, and $\beta(t)=\beta_R(t)\cup S'$ if $t\in V(R)$.
  Note that this is indeed a tree decomposition of $G[X]$.
  Furthermore, there is a node $r$ of $R$ such that $v\in \beta_R(r)$ and thus $S\subseteq \beta(r)$.

  Finally, let us consider the tree decomposition $\widehat{\mathcal{T}}=(T,\widehat{\beta})$ of $G[X]-S$ obtained by removing the vertices of $S$ from each bag of $\mathcal{T}$.
  Then for each node $t\in V(T')$, it holds that $\widehat{\beta}(t)\subseteq (\beta'(t)\setminus S')\cup \{u\}$, hence,
  $|\widehat{\beta}(t)|\le (|S|-1)+1= |S|$ by the induction hypothesis; and for each node $t\in V(R)$, it holds that $\widehat{\beta}(t)\subseteq \beta_R(t)$ and, hence, $|\widehat{\beta}(t)|\le |\beta_R(t)|\le 2\le |S|$.
  Thus, $\widehat{\mathcal{T}}$ has width at most $|S|-1$.
  This proves the lemma.
\end{proof}

The previous \lcnamecref{small tw} has the following consequence for induced minor models.

\begin{theorem}\label{minimal bags have small tw}
  Let $H$ and $G$ be two graphs such that $G$ contains $H$ as an induced minor.
  Let $\X = \{X_u \colon u \in V(H)\}$ be a minimal induced minor model of $H$ in $G$.
  Then, for every $u \in V(H)$, the graph $G[X_u]$ has treewidth at most $\max \{d_H(u)-1, 0\}$.
\end{theorem}

\begin{proof}
  Fix a vertex $u \in V(H)$.
  If $d_H(u) = 0$, then $|X_u| = 1$ by the minimality of $\X$, hence, $\tw(G[X_u]) = 0$.
  Assume now that $d_H(u) \geq 1$.
  For each neighbor $v$ of $u$ in $G$, let $x_v$ be a vertex in $X_v$ that is adjacent in $G$ to some vertex in $X_u$, and let $S = \{x_v\colon v\in N_H(u)\}$.
  Let $G'$ be the graph obtained from $G[X_u\cup S]$ by deleting the edges with both endpoints in $S$ (if any).
  Then $X_u\cup S$ is a set inducing a connected subgraph of $G'$ that contains $S$ and is minimal with respect to this property.
  By \cref{small tw}, there is a tree decomposition $\T$ of $G'$ of width at most $|S|$ such that one bag contains $S$ and the tree decomposition obtained from $\T$ by removing vertices in $S$ from the bags is a tree decomposition of $G'[X_u]$ of width at most $|S|-1$.
  Since $G'[X_u] = G[X_u]$, this implies that the graph $G[X_u]$ has treewidth at most $d_H(u)-1$, as claimed.
\end{proof}

Note that the bound on the treewidth of $G[X_w]$ in \cref{minimal bags have small tw} is tight.
For instance, for the graph $G_n$ consisting of a clique on $n$ vertices to which exactly one leaf is attached on each vertex and where $H_n$ is the star with $n$ leaves.
More formally, for every positive integer $n$, consider the graph $G_n$ with vertex set $V(G_n) = \{u_1,\ldots, u_n\}\cup\{v_1,\ldots, v_n\}$ and edge set $\{\{u_i,u_j\}\colon 1\le i<j\le n\}\cup \{\{u_i,v_i\}\colon 1\le i\le n\}$.
Furthermore, let $H_n$ denote the star $K_{1,n}$, that is, the graph with vertex set $V(H_n) = \{a,b_1,\ldots, b_n\}$ and edge set $\{\{a,b_i\}\colon 1\le i\le n\}$.
Let $X_a = \{u_1,\ldots, u_n\}$ and $X_{b_i} = \{v_i\}$ for all $i\in \{1,\ldots, n\}$.
Then, $\X = \{X_u \colon u \in V(H_n)\}$ is a minimal induced minor model of $H_n$ in $G_n$.
Furthermore, the graph $G_n[X_a]$ is the complete graph $K_n$ and therefore has a treewidth equal to $n-1 = d_{H_n}(a)-1$.

\subsection{On the treewidth of minimal models}\label{subsec:tw-model}

\Cref{minimal bags have small tw} motivates the question of whether the treewidth of the subgraph induced by the entire induced minor model can also be bounded.
We show next that this statement fails in general, even when $H$ is a tree.

\begin{figure}[htp]
  \centering
  \begin{tikzpicture}[scale=1.1]
    \tikzset{every node/.style={v:main}}
    \begin{scope}[xscale=0.75]
      \node[v:main,fill=red,fill opacity=0.1,label={below left:$p'_1$}] (L1) at (-1,0) {};
      \node[v:main,label={above:$y'_1$}] (l1) at (-1,1) {};
      \node[v:main,label={above:$x'_1$}] (l2) at (-2,0) {};
      \node[v:main,label={below:$z'_1$}] (l3) at (-1,-1) {};
      \node[v:main,label={left:$w'_1$}] (ll) at (-3,0) {};
      \node[v:main,fill=red,fill opacity=0.1,label={below left:$p'_2$}] (R1) at (1,0) {};
      \node[v:main,label={above:$y'_2$}] (r1) at (1,1) {};
      \node[v:main,label={above:$x'_2$}] (r2) at (2,0) {};
      \node[v:main,label={below:$z'_2$}] (r3) at (1,-1) {};
      \node[v:main,label={right:$w'_2$}] (rr) at (3,0) {};

      \foreach \x in {1,2,3}{
          \draw (L1) to (l\x);
          \draw (R1) to (r\x);
        }
      \draw (L1) to (R1)
      (ll) to (l2)
      (rr) to (r2);
    \end{scope}

    \begin{scope}[xshift=6.5cm,xscale=0.75,yscale=0.5]
      \node[v:main,label={above left:$x_1$}] (L) at (-2,0) {};
      \node[v:main,label={left:$w_1$}] (VL) at (-3,0) {};
      \node[v:main,label={above right:$x_2$}] (R) at (2,0) {};
      \node[v:main,label={right:$w_2$}] (VR) at (3,0) {};
      \draw (VL) to (L)
      (VR) to (R);
      \node[v:main,label={above:$y_1$}] (TL4) at (-1,4) {};
      \node[v:main,label={below:$z_1$}] (BL4) at (-1,-4) {};
      \node[v:main,label={above:$y_2$}] (TR4) at (1,4) {};
      \node[v:main,label={below:$z_2$}] (BR4) at (1,-4) {};

      \foreach \x in {3,2,1} {
          \node[v:main] (TL\x) at (-1,\x) {};
          \node[v:main] (BL\x) at (-1,-\x) {};
          \node[v:main] (TR\x) at (1,\x) {};
          \node[v:main] (BR\x) at (1,-\x) {};
        }
      \coordinate (ML1) at (-1,-0.25);
      \coordinate (ML2) at (-1,0.25);
      \coordinate (MR1) at (1,-0.25);
      \coordinate (MR2) at (1,0.25);
      \foreach \P in {TL,BL,TR,BR}{
          \draw (\P4) to (\P3)
          (\P3) to (\P2)
          (\P2) to (\P1);
        }
      \draw[dotted] (TL1) to (BL1);
      \draw[dotted] (TR1) to (BR1);

      \foreach \x in {TL1,TL2,TL3,BL1,BL2,BL3,ML1,ML2}{
          \draw (TR1) to (\x)
          (TR2) to (\x)
          (TR3) to (\x)
          (BR1) to (\x)
          (BR2) to (\x)
          (BR3) to (\x);
        }
      \foreach \x in {TL1,TL2,TL3,BL1,BL2,BL3}{
          \draw (\x) to (MR1)
          (\x) to (MR2);
        }
      \draw (L) to (TL3)
      (L) to (TL2)
      (L) to (TL1)
      (L) to (ML1)
      (L) to (ML2)
      (L) to (BL1)
      (L) to (BL2)
      (L) to (BL3);
      \draw (R) to (TR3)
      (R) to (TR2)
      (R) to (TR1)
      (R) to (MR1)
      (R) to (MR2)
      (R) to (BR1)
      (R) to (BR2)
      (R) to (BR3);
      \fill[red,opacity=0.1] (-1,0) ellipse (0.55cm and 3.5cm);
      \draw[dashed] (-1,0) ellipse (0.55cm and 3.5cm);
      \node[v:empty,label={$P_1$}] at (-1.85,-2.5) {};
      \fill[red,opacity=0.1] (1,0) ellipse (0.55cm and 3.5cm);
      \draw[dashed] (1,0) ellipse (0.55cm and 3.5cm);
      \node[v:empty,label={$P_2$}] at (+1.85,-2.5) {};
    \end{scope}

  \end{tikzpicture}
  \caption{The graphs $T$ (left) and $G_k$ (right) from \cref{H with large treewidth}.
  }
  \label{fig:H large tw}
\end{figure}

\begin{proposition}\label{H with large treewidth}
  There exists a $10$-vertex tree $T$ such that, for every positive integer~$k$, there exists a graph $G_k$ with $2k+8$ vertices in which every induced minor model of $T$ induces a subgraph of degeneracy at least $k$, and thus of treewidth at least~$k$.
\end{proposition}

\begin{proof}
  Let $T$ be the $10$-vertex tree depicted in \cref{fig:H large tw}.
  For any integer $k \geq 1$, let $G_k$ be the graph obtained as follows.
  Start with two vertex-disjoint paths $\widehat{P}_1$ and $\widehat{P}_2$ with $k+2$ vertices each, and for each $i \in \{1,2\}$, let $y_i,z_i$ be the endpoints of $\widehat{P}_i$, and $P_i$ the subpath of $\widehat{P}_i$ induced by its internal vertices.
  (Note that $P_i$ is nonempty, since $k\ge 1$.)
  Then, add all the edges so that $P_1$ is complete to $P_2$.
  Add four vertices $w_1$, $x_1$, $w_2$, and $x_2$ such that $w_ix_i$ is an edge, for $i \in \{1,2\}$.
  Finally, connect $x_1$ to all vertices of $P_1$, and $x_2$ to all vertices of $P_2$.
  The graph $G_k$ has $2k+8$ vertices.
  See \cref{fig:H large tw} for a schematic representation of~$G_k$.

  Since the degeneracy of any graph is bounded from above by its treewidth, it suffices to show that every induced minor model of $T$ in $G_k$ induces a subgraph of degeneracy at least $k$.
  Note that $P_1$ and $P_2$ each contain $k$ vertices and are complete to each other.
  Therefore, any induced subgraph of $G_k$ containing the vertices in $V(P_1)\cup V(P_2)$ has degeneracy at least $k$.

  The graph $G_k$ contains $T$ as an induced minor, obtained by contracting the paths $P_1$ and $P_2$, respectively, into single vertices.
  The corresponding induced minor model consists of branch sets each containing one vertex, except possibly for the two branch sets equal to $V(P_1)$ and $V(P_2)$. %
  From our previous arguments, this model has degeneracy at least $k$.
  We claim that this is the unique induced minor model of $T$ in $G_k$, up to symmetry.

  Let $\X = \{X_u\colon u\in V(T)\}$ be an induced minor model of $T$ in $G_k$.
  Since $T$ contains an induced $6$-vertex path $W_T$ between $w'_1$ and $w'_2$, the graph $G_k$ must admit a chordless path on at least $6$ vertices containing vertices from each of the $6$ branch sets of $\X$ corresponding to vertices of $W_T$.
  Denote this path in $G_k$ by $W$.

  Suppose first that $W$ contains two vertices from $P_1$.
  We claim that $W$ does not contain any vertex of $V(P_2)\cup \{x_1\}$.
  Indeed, $W$ can contain at most one such vertex $u$, otherwise $W$ would contain a cycle, and such a vertex $u$ must be an internal vertex of $W$.
  So all the vertices of $W$ but $u$ have to be in $\widehat{P_1}$ and all the vertices of $P_1$ are adjacent to $u$, a contradiction.
  Hence, $W$ is included in $\widehat{P_1}$.
  For a similar reason, since the path $W_T$ is induced and $\X$ is an induced minor model of $T$ in $G_k$, the branch sets of the vertices in $W_T$ do not intersect $V(P_2)\cup \{x_1\}$.
  Therefore, the branch sets of the vertices of $W_T$ are included in $\widehat{P_1}$.
  Since the branch sets of $p_1'$ and $p_2'$ have the same neighborhood in $V(G_k) \setminus \widehat{P_1}$, the branch sets of $y_1',z_1',x_1'$ cannot contain a vertex of $P_2$ nor $x_1$; so they belong to $\widehat{P_1}$.
  This yields a contradiction, since the branch set of $p_1'$ only has two neighbors in $\widehat{P_1}$.
  This shows that $W$ contains at most one vertex from $P_1$.

  Now, if $W$ contains no vertex of $P_2$, then, since $W$ has at least $6$ vertices, $W$ has to contain two vertices of $P_1$, a contradiction.
  We conclude by symmetry that $W$ contains exactly one vertex from each $P_1$ and $P_2$.

  It can then be observed that the only possible option for $W$ is to have exactly $6$ vertices and to lie between $w_1$ and $w_2$ in $G_k$.
  By symmetry, we may assume without loss of generality that for each $v\in \{w_1,w_2,x_1,x_2\}$, we have $X_{v'} = \{v\}$.
  Consequently, for each $i \in \{1,2\}$, the unique vertex $p_i$ in $W\cap X_{p_i'}$ belongs to $P_i$.
  Moreover, since $p'_1$ is not adjacent to $x'_2$ in $T$, the branch set $X_{p'_1}$ has to be included in $V(\widehat{P}_1)$, and similarly, $X_{p'_2}\subseteq V(\widehat{P}_2)$.
  Then, the branch sets $X_{y'_1}$ and $X_{z'_1}$ cannot contain vertices adjacent to $p_2$ nor to $x_2$, so they cannot contain a vertex in $P_1$ nor $P_2$, and thus $X_{y'_1}\cup X_{z'_1} = \{y_1,z_1\}$.
  Similarly, $X_{y'_2}\cup X_{z'_2} = \{y_2,z_2\}$.
  By symmetry, we may assume without loss of generality that for each $v\in \{y_1,y_2,z_1,z_2\}$, we have $X_{v'} = \{v\}$.
  Finally, for each $i\in \{1,2\}$, to ensure that there is an edge between $X_{p_i'}$ and each of $X_{y_i'} = \{y_i\}$ and $X_{z_i'} = \{z_i\}$, the fact that $X_{p_i'}$ induces a connected subgraph implies that $X_{p_i'} = V(P_i)$.
  We thus obtain the induced minor model of $T$ as previously described.
\end{proof}

Observe that the tree $T$ from \Cref{H with large treewidth} admits a unique edge with both endpoints of degree larger than $2$.
Next, we show that, for any graph $H$ without such edges, any graph $G$ containing $H$ as an induced minor admits an induced minor model of $H$ with treewidth at most $|V(H)|$, and thus degeneracy at most $|V(H)|$.
This strengthens an analogous result due to Korhonen and Lokshtanov \cite[Lemma~5.2]{korhonen2023inducedminorfree} establishing a bound on degeneracy of $3|V(H)|$.

\begin{proposition}\label{subdivided edges mean small tw}
  Let $H$ be a graph every edge of which is incident to a vertex of degree at most $2$, and let $G$ be a graph that contains $H$ as an induced minor.
  Then, every minimal model of $H$ in $G$ has treewidth at most $|R|$, %
  where $R$ is the set of vertices with degree at most $2$ in $H$.
\end{proposition}

\begin{proof}
  Let $\X = \{X_u \colon u \in V(H)\}$ be a minimal induced minor model of $H$ in $G$ satisfying the conditions of
  \cref{lem: min model deg 2}.
  Let $G_H = G[\X]$ and let $R \subseteq V(H)$ be the set of vertices of $H$ that have degree at most~$2$.
  It is easy to see that if $H$ is a cycle, then $G_H$ is a cycle, which has treewidth $2\leq |R|-1$ and the result holds.
  Moreover, if $H$ is not connected, then if suffices to prove the statement on each connected component.
  We can thus assume that $H$ is connected and is not a cycle.

  \Cref{lem: min model deg 2} guarantees that, for every $r\in R$, the branch set $X_r$ contains a unique vertex.
  Let $S = \bigcup_{r\in R}X_r$. Note that we have $|S|=|R|$.
  Moreover, since every edge of $H$ is incident to a vertex of degree at most $2$, every connected component of $G_H-S$ is induced by a unique branch set $X_v$, for some $v \in V(H)\setminus R$.

  Let $G'$ be the graph obtained from $G_H$ by removing every edge with both endpoints in $S$, and notice that $G_H-S=G'-S$. %
  We construct a tree decomposition $\mathcal{T}$ of $G'$ of width at most $|S|$ such that there is a bag containing $S$. As a consequence, $\mathcal{T}$ is a tree decomposition of $G_H$, which gives the desired result.

  For every connected component $G[X_v]$ of $G'-S$, for $v \in V(H)\setminus R$, let $S_v\subseteq S$ be the neighborhood of $X_v$ in $G'$.
  In particular, for each $v \in V(H)\setminus R$, it holds that $S_v$ is an independent set in $G'$.
  Moreover, since $\X$ is minimal, the set $X_v \cup S_v$ induces a connected subgraph of $G'$ that contains $S_v$ and is minimal with respect to this property.
  Thus, for each $v \in V(H)\setminus R$, \cref{small tw} applies to the induced subgraph $G_v=G'[X_v\cup S_v]$.

  Now, we construct a tree decomposition $\T = (T, \beta)$ of $G'$ as follows.
  For each component $G[X_v]$ of $G'-S$, for $v \in V(H)\setminus R$, we define $G_v$ as above.
  Then \cref{small tw} guarantees the existence of a tree decomposition $\T_v=(T_v,\beta_v)$ of $G_v$ of width at most $|S_v|\leq |S|$, with a node $t_v\in V(T_v)$ such that $S_v\subseteq \beta_v(t_v)$.
  Let $T$ be any tree obtained from the disjoint union of all trees $T_v$ by adding a node $t_S$ adjacent all the nodes $t_v$, and $\beta$ is defined so that for each $t \in V(T)$, $\beta(t) = \beta_v(t)$ if $t \in V(T_v)$, and $\beta(t_S) = S$.
  The resulting tree decomposition $\T = (T, \beta)$ of $G'$ is also a tree decomposition of $G_H$ and has width at most $|S|$, which completes the proof.
\end{proof}

\subsection{Minimal models have small chromatic number}\label{subsec:omega-model}

We conclude this section by proving the following bound on the chromatic number of a minimal induced minor model.

\begin{theorem}\label{model with bounded chi}
  Let $H$ and $G$ be two graphs such that $G$ contains $H$ as an induced minor.
  Then, for any minimal induced minor model $\X$ of $H$ in $G$, it holds $\chi(G[\X]) \leq \max\{\min\{2|E(H)|, \chi(H)\Delta(H)\},1\}$.

\end{theorem}
\begin{proof}
  Let $\X = \{X_u\colon u\in V(H)\}$ be a minimal induced minor model of $H$ in $G$, and let $G' = G[\X]$ be the subgraph of $G$ induced by $\X$.
  We may assume that $H$ is connected, since if $H$ is disconnected, the branch sets of different components are anticomplete to each other, and it suffices to prove the claim for every connected component of $H$.

  If a vertex $u\in V(H)$ has degree $0$, $H$ is a single vertex and $|X_u| = 1$, hence, $\chi(G') = 1$ and the conclusion follows.
  So we can assume that all the vertices of $H$ have degree at least $1$.

  Let $I_1,\dots, I_{\chi(H)}$ be a partition of the vertices of $H$ into $\chi(H)$ independent sets.
  Let $1 \le j\le \chi(H)$. Let $G_j$ be the subgraph of $G'$ induced by the union of the branch sets of the vertices in $I_j$, that is, $V(G_j) = \bigcup_{u\in I_j}X_u$.
  For each $u \in V(H)$,
  \begin{equation}\label{eq0}
    \chi(G[X_u])\leq \tw(G[X_u])+1\leq d_H(u)\,,
  \end{equation}
  where the first inequality holds since every graph $F$ satisfies $\chi(F) \leq \tw(F)+1$ and the second one holds by \Cref{minimal bags have small tw} (and since $d_H(u) \ge 1$).
  Since every connected component of $G_j$ corresponds to a unique branch set, we can use the same set of colors for each of those branch sets.
  Hence, $\chi(G_j) = \max \{\chi(G[X_u])\colon u\in I_j\}\le \max\{d_H(u)\colon u\in I_j\}$ by \cref{eq0}.
  We thus get the following upper bound on the chromatic number of $G'$:
  \begin{equation}\label{eq1}
    \chi(G')\leq \sum\limits_{j=1}^{\chi(H)}\chi(G_j)\leq \sum\limits_{j=1}^{\chi(H)} \max\{d_H(u)\colon u\in I_j\}
  \end{equation}

  To conclude, it is enough to show that the right term of \cref{eq1} is at most $\min\{2|E(H)|,\chi(H)\Delta(H)\}$.
  Indeed, for each $j$, $ \max\{d_H(u)\colon u\in I_j\}\leq \sum_{u\in I_j}d_H(u)$, and hence
  the sum of those maximum degrees is at most the sum of all the degrees of vertices in $H$, that is, $2|E(H)|$.
  Similarly, for each $j$, we have that  $ \max\{d_H(u) : u\in I_j\}\leq \Delta(H)$, and thus the righthand sum is at most $ \chi(H)\Delta(H)$.
\end{proof}

Let us note that the two bounds of $2|E(H)|$ and $\chi(H)\Delta(H)$ are incomparable to each other, as seen by the following examples.
\begin{itemize}
  \item Fix an integer $k\ge 1$ and let $H$ be a bipartite graph with maximum degree~$k$.
        Then, $\chi(H)\Delta(H)\le 2k$, while the value of $2|E(H)|$ can be arbitrarily large.
  \item For an integer $k\ge 2$, let $H_k$ be graph obtained from the star $K_{1,k}$ after adding an edge between two leaves.
        Then $\chi(H)\Delta(H) = 3k$ and $2|E(H)| = 2(k+1)$.
        Hence, for all $k\ge 3$, the bound of $2|E(H)|$ is better.\footnote{One can remark that we can increase the gap between $\chi(H)\Delta(H)$ and $2|E(H)|$ by replacing one of the leaves by a clique of size $\sqrt{k}$, but the order of magnitude of the gap cannot be increased further.}
\end{itemize}

\section{Induced minors and bounded tree-independence number}
\label{sec:imt-tree-independence}

Recall that there exist graphs $H$ for which the problem of deciding whether $H$ is an induced minor of an input graph $G$ is \NP-complete.
In stark contrast, in this section we show that, for any graph class $\mathcal{G}$ with bounded tree-independence number and any fixed graph $H$, the problem is solvable in polynomial time.

Let us first remark that for a fixed graph $H$, the property that a graph $G$ admits $H$ as an induced minor can be expressed as an $\mathsf{MSO}_1$ formula.\footnote{$\mathsf{MSO}_1$ logic allows for quantifications over vertices and vertex subsets of graphs, using the standard Boolean connectives and equality testing.}
Indeed, we simply look for $|V(H)|$ pairwise disjoint sets of vertices of $G$, each inducing a connected subgraph, such that two sets share an edge if and only if the corresponding vertices of $H$ are adjacent.
All of these properties can be expressed in $\mathsf{MSO}_1$, and thus in $\mathsf{CMSO}_2$, which is an extension of $\mathsf{MSO}_1$ (see~\cite{courcelle2012graph}).
In particular, the following result applies with the $\mathsf{MSO}_1$ formula $\psi$ corresponding to the containment of $H$ as an induced minor.

\begin{theorem}[Lima et al.~\cite{LMMORS24}]\label{meta theorem tin}
  For every $k,r$ and a $\mathsf{CMSO}_2$ formula $\psi$, given a graph $G$ with tree-independence number at most $k$ and a weight function $\mathfrak{w} : V(G) \to \mathbb{Q}_+$, we can find in polynomial time a set $Y \subseteq V(G)$ such that
  \begin{enumerate}
    \item $G[Y] \models \psi$,
    \item $\omega(G[Y]) \leq r$,
    \item $Y$ is of maximum weight subject to the conditions above,
  \end{enumerate}
  or conclude that no such set exists.
\end{theorem}

We can now show the main result of this section.

\begin{theorem}\label{detecting H in bounded tree-alpha}
  For every fixed graph $H$ and integer $k$, given a graph $G$ with tree-independence number at most $k$, there is a polynomial-time algorithm that either returns a minimal induced minor model of $H$ in $G$ or determines that $G$ does not contain $H$ as an induced minor.
\end{theorem}

\begin{proof}
  It is known that the class of graphs with tree-independence number at most $k$ is closed under taking induced minors (see~\cite{zbMATH07796423}), so by \cref{lem algo}, it suffices to show that we can decide in polynomial time whether $G$ contains $H$ as an induced minor.

  Note that if $G$ contains $H$ as an induced minor, then by \cref{model with bounded chi}, every minimal induced minor model of $H$ has chromatic number at most $c_H \coloneqq \max\{\min\{2|E(H)|, \allowbreak \chi(H)\Delta(H)\},1\}$ and then clique number at most $c_H$.
  By \cref{meta theorem tin} applied with $r=c_H$ and the formula $\psi$ checking the property of being an induced minor model of $H$, and $\mathfrak{w} \colon V(G) \to \{0\}$ (indeed, $\mathfrak{w}$ can be any function with codomain $\mathbb{Q}_{+}$).
  If the algorithm returns a set $Y$, then $Y$ corresponds to the union of the branch sets of a model of $H$ in $G$, and thus $G$ contains $H$ as an induced minor.
  Otherwise, if such a set $Y$ does not exist, then $G$ is $H$-induced-minor-free.
\end{proof}

\begin{remark}
  Let us denote, for a positive integer $k$, by $\mathcal{G}_k$ the class of graphs with tree-independence number at most $k$.
  Since each class $\mathcal{G}_k$ is closed under taking induced minors, it can be characterized by a unique family $\mathcal{F}_k$ of minimal induced minors.
  While $\mathcal{F}_1 = \{C_4\}$ (see~\cite{zbMATH07796423}), it is not known if for $k>1$ the sets $\mathcal{F}_k$ are finite or not.
  In this regard, we observe that for every $k\ge 4$, \Cref{detecting H in bounded tree-alpha} and the fact that the problem of recognizing graphs in  $\mathcal{G}_k$ is $\mathsf{NP}$-complete (see~\cite{DFGKM2024}), imply that, unless $\mathsf{P}=\mathsf{NP}$, the set $\mathcal{F}_k$ is infinite.
  Indeed, if for some $k\ge 4$ the set $\mathcal{F}_k$ is finite, then the problem of recognizing graphs in $\mathcal{G}_k$ is in $\mathsf{P}$, since it can be solved by the following polynomial-time algorithm.
  Given a graph $G$, we first invoke an algorithm due to Dallard et al.~from~\cite{DFGKM2024} to either conclude that the tree-independence number of $G$ is more than $k$ (in which case $G\not\in\mathcal{G}_k$), or obtain a tree decomposition of $G$ with independence number at most $8k$.
  In the latter case, since $G\in \mathcal{G}_k$ if and only if $G$ does not contain any of the graphs in $\mathcal{F}_k$ as an induced minor, and if $\mathcal{F}_k$ is finite, this latter condition can be tested in polynomial time due to \Cref{detecting H in bounded tree-alpha}.
\end{remark}

We can finally apply \Cref{detecting H in bounded tree-alpha} to graphs $H$ for which the class of $H$-induced-minor-free graphs has bounded tree-independence number.
Such graphs were characterized by Dallard, Milanič, and Štorgel~\cite{DMS_JCTB2024}, as follows.
We denote by $W_4$ the graph obtained the $4$-vertex cycle $C_4$ by adding to it a universal vertex and by $K_5^-$ the complete $5$-vertex graph minus an edge.
See \cref{fig:w4k5-k2q} for a representation of such graphs.

\begin{figure}
  \centering
  \begin{tikzpicture}[scale=0.55]
    \tikzset{every node/.style={draw,circle,fill=black,inner sep=0pt,minimum size=4.5pt}}

    \begin{scope}
      \node (a) at (0,0) {};
      \foreach \x/\y in {1/b,2/c,3/d,4/e}{
          \draw (\x*360/4: 2cm) node (\y) {};
        }
      \draw (a) to (b);
      \draw (a) to (c);
      \draw (a) to (d);
      \draw (a) to (e);
      \draw (b) to (c);
      \draw (c) to (d);
      \draw (d) to (e);
      \draw (e) to (b);
      \node[draw=none,fill=none,rectangle,label={$W_4$}] at (0,-4) {};
    \end{scope}
    \begin{scope}[xshift=7.5cm]
      \foreach \x/\y in {1/a,2/b,3/c,4/d,5/e}{
          \draw (+18+\x*360/5: 2cm) node (\y) {};
        }
      \foreach \x in {b,c,d,e}{
          \draw (a) to (\x);
        }
      \foreach \x in {c,d}{
          \draw (b) to (\x);
        }
      \foreach \x in {d,e}{
          \draw (c) to (\x);
        }
      \draw (d) to (e);
      \node[draw=none,fill=none,rectangle,label={$K_5^-$}] at (0,-4) {};
    \end{scope}
    \begin{scope}[xshift=15cm]
      \node (l) at (-2,0) {};
      \node (r) at (2,0) {};
      \foreach \y in {-2,-1,0,1,2}{
          \node (m\y) at (0,\y) {};
          \draw (l) -- (m\y) -- (r);
        }
      \node[draw=none,fill=none,rectangle,label={$K_{2,5}$}] at (0,-4) {};
    \end{scope}
  \end{tikzpicture}
  \caption{The graphs $W_4$, $K_5^-$, and $K_{2,5}$.}
  \label{fig:w4k5-k2q}
\end{figure}

\begin{theorem}[Theorem 7.3 and Remarks 3.12, 5.5, and 6.5 in~\cite{DMS_JCTB2024}]\label{W4K5-K2q im-free bounded tin}
  Let $H$ be a graph.
  Then the class of $H$-induced-minor-free graphs has bounded tree-independence number if and only if $H$ is an induced minor of $W_4$, $K_5^-$, or $K_{2,q}$ for some $q \in \mathbb{Z}_{\ge 0}$.
  Furthermore, for each such graph $H$, there is a polynomial-time algorithm that, given a graph $G$, either correctly determines that $G$ contains $H$ as an induced minor or outputs a tree-decomposition of $G$ with independence number at most $k_H$, where $k_H= 2(q-1)$ if $H$ is an induced minor of $K_{2,q}$, for some $q\ge 2$, and $k_H = 4$ if $H$ is an induced minor of $W_4$ or $K_5^-$.
\end{theorem}

\begin{corollary}\label{detecting W4K5k2q}
  Let $H$ be an induced minor of $W_4$, $K_5^-$, or $K_{2,q}$, for some $q \geq 0$.
  Then, given a graph $G$, there is a polynomial-time algorithm that either returns a minimal induced minor model of $H$ in $G$ or determines that $G$ does not contain $H$ as an induced minor.
\end{corollary}

\begin{proof}
  By \cref{lem algo}, since the class of all graphs is trivially closed under taking induced minors, it suffices to show that we can decide in polynomial time whether $G$ contains $H$ as an induced minor.
  By \Cref{W4K5-K2q im-free bounded tin}, there exists an integer $k_H$ and a polynomial-time algorithm $\mathcal{A}$ that, given a graph $G$, either correctly determines that $G$ contains $H$ as an induced minor or outputs a tree-decomposition of $G$ with independence number at most $k_H$.
  Let $G$ be a graph.
  Run the algorithm $\mathcal{A}$ on $G$.
  Then, either $\mathcal{A}$ correctly determines that $G$ contains $H$ as an induced minor, or $G$ has tree-independence number at most $k_H$, in which case \cref{detecting H in bounded tree-alpha} applies.
\end{proof}

\section{A structural approach for the case of \texorpdfstring{$K_{2,3}$}{K2,3}}
\label{sec:k23}

In this section, we provide an alternative algorithm to detect $K_{2,3}$ as an induced minor.
The algorithm relies on a characterization of $K_{2,3}$-induced minor-free graphs in terms of excluding particular induced subgraphs, called Truemper configurations, and on what is called a \emph{shortest path detector}, a method originally invented for the detection of pyramids in~\cite{MR2127609}.
More precisely, by relying on several previous works, we can easily detect three of the four Truemper configurations identified in \cref{sec:r} (namely the 3-path-configurations known as the \emph{thetas}, \emph{pyramids}, and \emph{long prisms}, the definitions are given later).
For instance, detecting $K_{2,3}$ as an induced topological minor can be equivalently stated as detecting a Truemper configuration called the \emph{theta}, a problem that can be solved in polynomial time~\cite{MR2728495,MR4141839}.
Furthermore, detecting a pyramid can also be done in polynomial time~\cite{MR2127609,MR4141839}, and we show how to detect a long prism in a pyramid-free graph by adapting an algorithm from~\cite{MR2191280}.
This is all done in \cref{sec:g}.

In \cref{sec:bw}, we give an algorithm to detect the last configuration identified in \cref{sec:r}, namely the \emph{broken wheel}, provided that the graph does not contain any of the 3-path-configurations detected in \cref{sec:g}, thus completing our algorithm.
Note that detecting a broken wheel is an \NP-complete problem, even when restricted to bipartite graphs~\cite{MR3289471} (see \cref{sec:bw} for details).

\subsection{Reducing to Truemper configurations}\label{sec:r}

In this section, we show that detecting $K_{2,3}$ as an induced minor is in fact equivalent to detecting whether some specific graphs are contained as induced subgraphs.
First, we give some definitions.
\medskip

A \emph{prism} is a graph made of three vertex-disjoint chordless paths $P_1 = a_1 \dots b_1$, $P_2 = a_2 \dots b_2$, $P_3 = a_3 \dots b_3$ of length at least 1, such that $\{a_1,a_2,a_3\}$ and $\{b_1,b_2,b_3\}$ are triangles and no edges exist between the paths except those of the two triangles.
A prism is \emph{long} if at least one of its three paths has length at least~2.

A \emph{pyramid} is a graph made of three chordless paths $P_1 = a \dots b_1$, $P_2 = a \dots b_2$, $P_3 = a \dots b_3$ of length at least~1, two of which have length at least 2, vertex-disjoint except at $a$, and such that $\{b_1,b_2,b_3\}$ is a triangle and no edges exist between the paths except those of the triangle and the three edges incident to~$a$.
We call $a$ the \emph{apex} of the pyramid.

A \emph{theta} is a graph made of three internally vertex-disjoint
chordless paths $P_1 = a \dots b$, $P_2 = a \dots b$, $P_3 = a \dots b$ of length at least~2 and such that no edges exist between the paths except the three edges incident to $a$ and the three edges incident to $b$. We call $a$ and $b$ the \emph{apices} of the theta.
Prisms, pyramids, and thetas are referred to as \emph{3-path-configurations}.

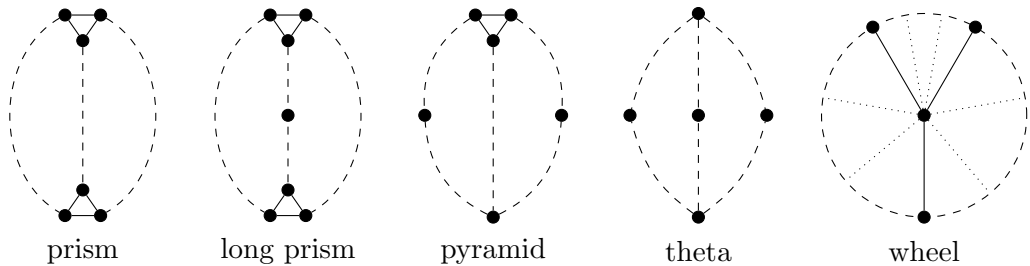
\begin{figure}[tb]
  \centering%
  \begin{tikzpicture}[scale=0.85]
    \tikzset{every node/.style={draw,circle,fill=black,inner sep=0pt,minimum size=4.5pt}}
    \begin{scope}[xshift=-7.5cm] %
      \coordinate (a) at (-1,0) {};
      \coordinate (b) at (0,0) {};
      \coordinate (c) at (1,0) {};
      \node (xa) at (100:1.5) {};
      \node (xb) at (90:1.1) {};
      \node (xc) at (80:1.5) {};
      \draw (xa) to (xb)
      (xb) to (xc)
      (xc) to (xa);
      \node (ya) at (80:-1.5) {};
      \node (yb) at (90:-1.1) {};
      \node (yc) at (100:-1-.5) {};
      \draw (ya) to (yb)
      (yb) to (yc)
      (yc) to (ya);
      \draw[dashed] (xa) to[bend right=60] (ya);
      \draw[dashed] (xb) to (yb);
      \draw[dashed] (xc) to[bend left=60] (yc);
      \node[draw=none,fill=none,rectangle] at (0,-2) {prism};
    \end{scope}
    \begin{scope}[xshift=-4.5cm] %
      \coordinate (a) at (-1,0) {};
      \node (b) at (0,0) {};
      \coordinate (c) at (1,0) {};
      \node (xa) at (100:1.5) {};
      \node (xb) at (90:1.1) {};
      \node (xc) at (80:1.5) {};
      \draw (xa) to (xb)
      (xb) to (xc)
      (xc) to (xa);
      \node (ya) at (80:-1.5) {};
      \node (yb) at (90:-1.1) {};
      \node (yc) at (100:-1-.5) {};
      \draw (ya) to (yb)
      (yb) to (yc)
      (yc) to (ya);
      \draw[dashed] (xa) to[bend right=60] (ya);
      \draw[dashed] (xb) to (yb);
      \draw[dashed] (xc) to[bend left=60] (yc);
      \node[draw=none,fill=none,rectangle] at (0,-2) {long prism};
    \end{scope}
    \begin{scope}[xshift=-1.5cm] %
      \node (a) at (-1,0) {};
      \coordinate (b) at (0,0) {};
      \node (c) at (1,0) {};
      \node (xa) at (100:1.5) {};
      \node (xb) at (90:1.1) {};
      \node (xc) at (80:1.5) {};
      \draw (xa) to (xb)
      (xb) to (xc)
      (xc) to (xa);
      \node (y) at (0,-1.5) {};
      \draw[dashed] (xa) to[bend right] (a);
      \draw[dashed] (xb) to (y);
      \draw[dashed] (xc) to[bend left] (c);
      \draw[dashed] (y) to[bend left] (a);
      \draw[dashed] (y) to[bend right] (c);
      \node[draw=none,fill=none,rectangle] at (0,-2) {pyramid};
    \end{scope}
    \begin{scope}[xshift=1.5cm] %
      \node (a) at (-1,0) {};
      \node (b) at (0,0) {};
      \node (c) at (1,0) {};
      \node (x) at (0,1.5) {};
      \node (y) at (0,-1.5) {};
      \draw[dashed] (x) to[bend right=20] (a);
      \draw[dashed] (x) to (b);
      \draw[dashed] (x) to[bend left=20] (c);
      \draw[dashed] (y) to[bend left=20] (a);
      \draw[dashed] (y) to (b);
      \draw[dashed] (y) to[bend right=20] (c);
      \node[draw=none,fill=none,rectangle] at (0,-2) {theta};
    \end{scope}
    \begin{scope}[xshift=4.8cm]
      \draw[dashed] circle (1.5) {};
      \node (x) at (0,0) {};
      \node (a) at (120:1.5) {};
      \node (b) at (60:1.5) {};
      \node (c) at (-90:1.5) {};

      \coordinate (ab1) at (100:1.5);
      \coordinate (ab2) at (80:1.5);
      \coordinate (bc1) at (10:1.5);
      \coordinate (bc2) at (-50:1.5);
      \coordinate (ac1) at (-140:1.5);
      \coordinate (ac2) at (-190:1.5);

      \foreach \y in {a,b,c}{
          \draw (x) to (\y);
        }
      \foreach \y in {ab1,ab2,bc1,bc2,ac1,ac2}{
          \draw[dotted] (x) to (\y);
        }
      \node[draw=none,fill=none,rectangle] at (0,-2) {wheel};
    \end{scope}
  \end{tikzpicture}
  \caption{A schematic representation of a prism, a long prism, a pyramid, a theta, and a wheel.
    Dashed edges represent paths of length at least $1$.
    Dotted edges may be present or not.}\label{f:tc}
\end{figure}

A \emph{hole} in a graph is a chordless cycle of length at least~4.
Observe that the lengths of the paths  in the definitions of prisms, pyramids, and thetas are designed so that the union of any two of the paths induces a hole.
A \emph{wheel} $W= (H,x)$ is a graph formed by a hole $H$ (called the \emph{rim}) together with a vertex $x$ (called the \emph{center}) that has at least three neighbors in the hole.
Prism, pyramids, thetas, and wheels are referred to as \emph{Truemper configurations}; see \cref{f:tc}.

A \emph{sector} of a wheel $(H,x)$ is a path $P$ that is contained in $H$, whose endpoints are neighbors of $x$ and whose internal vertices are not.
Note that $H$ is edgewise partitioned into the sectors of $(H,x)$.
Also, every wheel has at least three sectors.
A wheel is \emph{broken} if at least two of its sectors have length at least~2.

We denote by $\mathcal{S}$ the class of subdivisions of the graph $K_{1,3}$; see \cref{fig:S-T-M}.
The class $\mathcal{T}$ is the class of graphs that can be obtained from three paths of length at least one by selecting one endpoint of each path and adding three edges between those endpoints so as to create a triangle.
The class $\mathcal{M}$ is the class of graphs $H$ that consist of a path $P$ and a vertex $a$, called the \emph{center} of $H$, such that $a$ is nonadjacent to the endpoints of $P$ and $a$ has at least two neighbors in $P$.
Given a graph $H\in \mathcal{S}\cup\mathcal{T}\cup\mathcal{M}$, the \emph{extremities} of $H$ are the vertices of degree one as well as the center of $H$ in case $H\in\mathcal{M}$.
Observe that any $H\in \mathcal{S}\cup\mathcal{T}\cup\mathcal{M}$ has exactly three extremities.

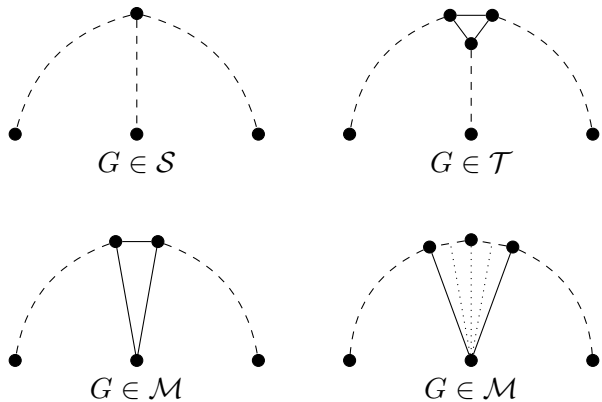
\begin{figure}[htb]
  \centering
  \begin{tikzpicture}[scale=0.8]
    \tikzset{every node/.style={draw,circle,fill=black,inner sep=0pt,minimum size=4.5pt}}
    \tikzset{decoration={amplitude=.5mm,segment length=2.25mm,post length=0mm,pre length=0mm}}
    \begin{scope}[xshift=-5.5cm]
      \node (b) at (2,0) {};
      \node (c) at (0,0) {};
      \node (d) at (-2,0) {};
      \node (a) at (0,2) {};

      \draw[dashed, bend left] (a) to (b);
      \draw[dashed] (a) to (c);
      \draw[dashed, bend right] (a) to (d);

      \node[rectangle,draw=none,fill=none] at (0,-.5) {$\strut G \in \mathcal{S}$};
    \end{scope}
    \begin{scope}[]
      \node (b) at (2,0) {};
      \node (c) at (0,0) {};
      \node (d) at (-2,0) {};
      \node (ac) at (90:1.5) {};
      \node (ab) at (80:2) {};
      \node (ad) at (100:2) {};

      \draw (ac) to (ab)
      (ac) to (ad)
      (ad) to (ab);

      \draw[dashed, bend left] (ab) to (b);
      \draw[dashed] (ac) to (c);
      \draw[dashed, bend right] (ad) to (d);

      \node[rectangle,draw=none,fill=none] at (0,-.5) {$\strut G \in \mathcal{T}$};
    \end{scope}
    \begin{scope}[xshift=-5.5cm,yshift=-3.75cm]
      \node (b) at (2,0) {};
      \node (c) at (0,0) {};
      \node (d) at (-2,0) {};
      \node (ab) at (80:2) {};
      \node (ad) at (100:2) {};

      \draw (ab) to (ad);

      \draw[dashed, bend left] (ab) to (b);
      \draw (c) to (ab)
      (c) to (ad);
      \draw[dashed, bend right] (ad) to (d);

      \node[rectangle,draw=none,fill=none] at (0,-.5) {$\strut G \in \mathcal{M}$};
    \end{scope}
    \begin{scope}[yshift=-3.75cm]
      \node (b) at (2,0) {};
      \node (c) at (0,0) {};
      \node (d) at (-2,0) {};
      \node (ac) at (90:2) {};
      \node (ab) at (70:2) {};
      \node (ad) at (110:2) {};

      \draw[dashed] (ab) to (ac)
      (ac) to (ad);
      \draw[dashed, bend left] (ab) to (b);
      \draw (c) to (ab)
      (c) to (ad);
      \draw[dashed, bend right] (ad) to (d);
      \draw[dotted] (ac) to (c);
      \coordinate (adc) at (80:2);
      \coordinate (abc) at (100:2);
      \draw[dotted] (c) to (adc)
      (c) to (abc);
      \node[rectangle,draw=none,fill=none] at (0,-.5) {$\strut G \in \mathcal{M}$};
    \end{scope}
  \end{tikzpicture}
  \caption{A schematic representation of graphs in $\mathcal{S}$, $\mathcal{T}$, and $\mathcal{M}$.
    Dashed edges represent paths of length at least $1$.
    Dotted edges may be present or not.}
  \label{fig:S-T-M}
\end{figure}

The key ingredient in our proof of \cref{t:k23T} is the following sufficient condition for the presence of an induced subgraph from $\mathcal{S} \cup \mathcal{T} \cup \mathcal{M}$ (which is also used later, in the proof of \cref{l:longP}).
This observation has been made many times in different contexts with slight variants; for instance, it is implicit in~\cite{watkinsMesner:cycle} and appears more explicitly in~\cite[Lemma~3.3]{maffray.t:reco} and \cite[Lemma 5.1]{MR4660624}.
For completeness, and to align with our notation, we provide a proof.

\begin{lemma}\label{l:con3}
  Let $G$ be a graph and $I$ be an independent set in $G$ with $|I| = 3$.
  If there exists a component $C$ of $G \setminus I$ such that $I \subseteq N(C)$, then there exists an induced subgraph $H$ of $G[N[C]]$ such that $H \in \mathcal{S} \cup \mathcal{T} \cup \mathcal{M}$ and $I$ is exactly the set of extremities of $H$.
\end{lemma}

\begin{proof}
  Let $I = \{a_1, a_2, a_3\}$ and assume that there exists a component $C$ of $G \setminus I$ such that $I \subseteq N(C)$.
  Let $P$ be a shortest path from $a_1$ to $a_2$ in $G[C\cup \{a_1,a_2\}]$.
  Consider a shortest path $P'$ in $G[C\cup \{a_3\}]$ from $a_3$ to the neighborhood of the interior of $P$.
    Let $u$ be the other endpoint of $P'$ (possibly $u = a_3$).
  Let $H$ be the subgraph of $G$ induced by $P\cup P'$.
  Note that $H$ is an induced subgraph of $G[C\cup \{a_1,a_2,a_3\}]$.

  Let $v_1$ (resp.\ $v_2$) be the neighbor of $u$ in $P$ closest to $a_1$ (resp.\ $a_2$) along $P$.

  Assume first that $u=a_3$.
  Then, $H$ belongs to $\mathcal S$ (if $v_1=v_2$) or to $\mathcal M$ (if $v_1\neq v_2$), in which case the center of $H$ is $a_3$.
  Note that in both cases, the set of extremities of $H$ is $\{a_1,a_2,a_3\} = I$.
  Hence, from here on we may assume that $u\neq a_3$, so $P'$ has length at least 1.

  If both $a_1$ and $a_2$ have neighbors in $P'$, let $w_i$ be the neighbor of $a_i$ in $P'$ closest to $a_3$ along $P'$ for $i\in \{1,2\}$.
    Without loss of generality, we can assume that in $P'$, vertex $w_1$ is at least as close to $a_3$ as $w_2$.
    Let $Q\coloneqq a_3P'w_2a_2$.
    Then $Q$ is a chordless path between $a_3$ and $a_2$ such that $a_1$ has at least one neighbor in the interior of $Q$, namely $w_1$.
    Hence, $Q\cup \{a_1\}$ induces a subgraph $H$ of $G$ that belongs to either $\mathcal{S}$ or $\mathcal{M}$, depending on whether $a_1$ has one neighbor or more in $Q$.
  Hence, from here on, we assume that $a_1$ has no neighbors in~$P'$.

  Assume that $a_2$ has neighbors in $P'$.
  Then, $Q\coloneqq a_1Pv_1uP'a_3$ is a chordless path  between $a_1$ and $a_3$ such that $a_2$ has at least one neighbor in the interior of $Q$.
    Therefore, $Q\cup\{a_2\}$ induces a subgraph $H$ of $G$ that belongs to either $\mathcal S$ or to $\mathcal M$, depending on whether $a_2$ has one neighbor or more in $Q$.
  Hence, from here on, we assume that $a_2$ has no neighbor in~$P'$.

  If $v_1=v_2$, then $H$ belongs to $\mathcal{S}$ and $I$ is the set of its  extremities.
  If $v_1v_2\in E(G)$, then $H$ belongs to $\mathcal{T}$, with triangle $\{u,v_1,v_2\}$ and the set of extremities $I$.
  So, we may assume that $v_1\neq v_2$ and $v_1v_2\notin E(G)$.
  In this case, $Q\coloneqq a_1Pv_1uv_2Pa_2$ is a chordless path between $a_1$ and $a_2$ such that $Q\cup P'$ induces a subgraph $H$ of $G$ that belongs to $\mathcal S$.
\end{proof}

\begin{theorem}\label{t:k23T}
  A graph contains $K_{2, 3}$ as an induced minor if and only if it contains a long prism, a pyramid, a theta, or a broken wheel as an induced subgraph.
\end{theorem}

\begin{proof}
  It is easy to check that a long prism, a pyramid, a theta, or a broken wheel all contain $K_{2,3}$ as an induced minor.

  Conversely, assume that a graph $G$ contains $K_{2,3}$ as an induced minor.
  Let us denote the vertices of $K_{2,3}$ as $\{u,v,a,b,c\}$ so that $\{u,v\}$ and $\{a,b,c\}$ form a bipartition of $K_{2,3}$.

  By \cref{lem: min model deg 2}, there exists $\X= \{X_u,X_v,X_a,X_b,X_c\}$ a minimal induced minor model of $K_{2,3}$ in $G$ such that $X_a$, $X_b$, and $X_c$ each contains only one vertex.
  Then $I = X_a\cup X_b \cup X_c$ is an independent set of size 3 in the subgraph $G'$ of $G$ induced by $\X$, that is, $G' = G[\X]$.
  Moreover, $I$ is contained in the neighborhood of each of the two connected components $X_u$ and $X_v$ of $G' \setminus I$.
  Consider the graphs $H = G[X_u \cup I]$ and $H' = G[X_v \cup I]$. Note in particular that $G'=G[V(H) \cup V(H')]$.
  By \cref{l:con3} and minimality of $\X$, the graphs $H$ and $H'$ belong to $\mathcal{S}\cup \mathcal{T}\cup \mathcal{M}$, since otherwise we could find a proper induced subgraph of one of them that is in $\mathcal{S}\cup \mathcal{T}\cup \mathcal{M}$, yielding an induced minor model of $K_{2,3}$ in $G$ smaller than~$\X$, contradicting the minimality of~$\X$.
  Moreover, the extremities of both $H$ and $H'$ are exactly $I=\{a_1,a_2,a_3\}$.

  \begin{claim}
    The graph $G' = G[V(H) \cup V(H')]$ contains a long prism, a pyramid, a theta, or a broken wheel as an induced subgraph.
  \end{claim}

  \begin{proofclaim}
    Assume first that $H \in \mathcal{M}$.
    W.l.o.g., we may assume that $a_2$ is the center of $H$.
    Let $P$ be the path from $a_1$ to $a_3$ in $H$ that does not contain $a_2$.
    Let $c_1$ and $c_3$ be the vertices of $P$ adjacent to $a_2$ that are closest to $a_1$ and $a_3$ along $P$, respectively.

    If $H' \in \mathcal{S}$, then let $u$ be the unique vertex of degree~3 in $H'$; see \cref{fig:cases-short-proof}.
    Observe that either $G'$ is a pyramid (if $c_1c_3\in E(G)$) or a broken wheel (if $c_1c_3\notin E(G)$ and $a_2u \in E(G)$), or $G'$ contains a theta with apices $a_2$ and $u$ (whenever $c_1c_3\notin E(G)$ and $a_2u\notin E(G)$).

    If $H' \in \mathcal{T}$, then notice that either $G'$ is a long prism (if  $c_1c_3\in E(G)$) or $G'$ contains a pyramid with apex $a_2$ (if  $c_1c_3\notin E(G)$).

    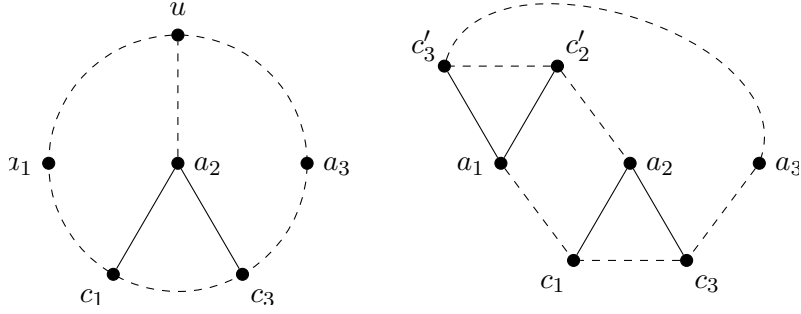
\begin{figure}
      \centering
      \begin{tikzpicture}[scale=0.85]
        \tikzset{every node/.style={draw,circle,fill=black,inner sep=0pt,minimum size=4.5pt}}
        \begin{scope}
          \draw[dashed] circle (2) {};

          \node[label=right:{$\strut a_2$}] (x) at (0,0) {};
          \node[label=left:{$\strut a_1$}] (y) at (180:2) {};
          \node[label=right:{$\strut a_3$}] (z) at (0:2) {};
          \node[label=below left:{$\strut c_1$}] (a) at (-120:2) {};
          \node[label=below right:{$\strut c_3$}] (b) at (-60:2) {};
          \node[label=above:{$\strut u$}] (c) at (90:2) {};

          \foreach \y in {a,b}{
              \draw (x) to (\y);
            }

          \draw[dashed] (x) to (c);

        \end{scope}
        \begin{scope}[xshift=7cm]
          \node[label=right:{$\strut a_2$}] (a2) at (0,0) {};
          \node[label=left:{$\strut a_1$}] (a1) at (180:2) {};
          \node[label=right:{$\strut a_3$}] (a3) at (0:2) {};

          \node[label=above left:{$\strut c_3'$}] (u) at ($(a1)+(120:1.75)$) {};
          \node[label=above right:{$\strut c_2'$}] (v) at ($(a1)+(60:1.75)$) {};

          \draw (a1) to (u)
          (a1) to (v);
          \draw[dashed] (v) to (u);

          \node[label=below left:{$\strut c_1$}] (u') at ($(a2)+(-120:1.75)$) {};
          \node[label=below right:{$\strut c_3$}] (v') at ($(a2)+(-60:1.75)$) {};

          \draw (a2) to (u')
          (a2) to (v');
          \draw[dashed] (v') to (u');

          \draw[dashed] (u) to[bend left=90] (a3);
          \draw[dashed](v) to (a2);
          \draw[dashed] (a1) to (u');
          \draw[dashed] (a3) to (v');
        \end{scope}
      \end{tikzpicture}
      \caption{Illustration of two cases in the proof of \cref{t:k23T}: on the left, $H \in \mathcal{M}$ and $H' \in \mathcal{S}$; on the right, $H, H' \in \mathcal{M}$.
        Dashed edges represent paths of length at least $1$.}
      \label{fig:cases-short-proof}
    \end{figure}

    If $H' \in \mathcal{M}$ and $H'$ is centered at $a_2$, then $G'$ is a broken wheel.
    So, assume w.l.o.g.\ that $H'$ is centered at $a_1$; see \cref{fig:cases-short-proof}.
    Let $P'$ be the path from $a_2$ to $a_3$ in $H'$ that does not contain $a_1$.
    Let $c_2'$ and $c_3'$ be the vertices of $P'$ adjacent to $a_1$ that are closest to $a_2$ and $a_3$ along $P'$, respectively.
    If $c_1c_3$ and $c'_2c'_3$ are both edges of $G$, then $G'$ is a long prism with triangles $\{a_2,c_1,c_3\}$ and $\{a_1,c'_2,c'_3\}$.
    So we may assume w.l.o.g.\ that $c_1c_3\notin E(G)$.
    If $c'_2c'_3\in E(G)$, then $G'$ contains a pyramid with apex $a_2$ and triangle $\{a_1,c'_2,c'_3\}$.
    If $c'_2c'_3\notin E(G)$, then $G'$ contains a theta with apices $a_1$ and $a_2$.

    Hence, we may assume that $H \notin \mathcal{M}$ and symmetrically $H' \notin \mathcal{M}$.

    Assume that $H \in \mathcal{S}$.
    If $H' \in \mathcal{S}$, then $G'$ is a theta and if $H' \in \mathcal{T}$, then $G'$ is a pyramid.
    Hence, we may assume that $H \notin \mathcal{S}$ and symmetrically $H' \notin \mathcal{S}$.

    We are left with the case where both $H$ and $H'$ belong to $\mathcal{T}$, which implies that $G'$ is a long prism.
  \end{proofclaim}
  This completes the proof of \cref{t:k23T}.
\end{proof}

\subsection{Detecting 3-path-configurations}
\label{sec:g}

The first polynomial-time algorithm to detect a pyramid was proposed by Chudnovsky and Seymour~\cite{MR2127609}.
A faster algorithm was obtained by Lai, Lu, and Thorup~\cite[Theorem 1.3]{MR4141839}.
The $\widetilde{O}(\cdot)$ notation omits polylogarithmic factors.

\begin{theorem}\label{th:pyramid}
  Determining whether a given graph contains a pyramid as an induced subgraph can be done in time $\widetilde{O}(n^5)$.
\end{theorem}

Detecting a theta in polynomial time was first done by Chudnovsky and Seymour~\cite{MR2728495}.
A faster method was later discovered by Lai, Lu, and Thorup~\cite[Theorem 1.2]{MR4141839}.

\begin{theorem}
  \label{th:theta}
  Determining whether a given graph contains a theta as an induced subgraph can be done in time $\widetilde{O}(n^6)$.
\end{theorem}

We now explain how a long prism can be detected in a graph with no pyramid.
Note that detecting a prism is known to be an \textsf{NP}-complete problem~\cite{MR2191280}.
Our algorithm is inspired from the algorithm in~\cite{MR2191280} to detect a pyramid or a prism.
Adapting it to detect a long prism in a pyramid-free graph is straightforward.
We include all details just for the sake of completeness.
For the reader's convenience, we tried to optimize the length of the proof instead of the running time of the algorithm, since anyway the complexity bottleneck is in the next section.

\begin{lemma}
  \label{l:longP}
  Determining whether a given pyramid-free graph contains a long prism as an induced subgraph can be done in time $\mathcal{O}(n^6(n+m))$.
\end{lemma}
\begin{proof}
  We call \emph{co-domino} the graph with vertex set $\{a_1, a_2, v_1, v_2, v_3, v_4\}$ and edges $v_1v_2$, $v_2v_3$, $v_3v_4$, $v_4v_1$, $a_1v_1$, $a_1v_2$, $a_2v_3$, and $a_2v_4$.
  We call \emph{net} the graph with vertex set $\{a_1, a_2, a_3, v_1, v_2, v_3\}$ and edges $v_1v_2$, $v_2v_3$, $v_3v_1$, $a_1v_1$, $a_2v_2$, and $a_3v_3$.
  See \cref{co-domino and net figure} for a representation of the co-domino and the net.

  \begin{figure}
    \centering
    \begin{tikzpicture}[scale=1]
      \tikzset{every node/.style={draw,circle,fill=black,inner sep=0pt,minimum size=4.5pt}}
      \begin{scope}[xshift=0cm,yshift=.5cm]
        \node[label=left:{$\strut a_1$}] (a) at (-2,0) {};
        \node[label=above left:{$\strut v_2$}] (b) at (-1,1) {};
        \node[label=below left:{$\strut v_1$}] (c) at (-1,-1) {};
        \node[label=right:{$\strut a_2$}] (a') at (2,0) {};
        \node[label=above right:{$\strut v_3$}] (b') at (1,1) {};
        \node[label=below right:{$\strut v_4$}] (c') at (1,-1) {};
        \draw (a) to (b)
        (a) to (c)
        (b) to (b')
        (b) to (c)
        (c) to (c')
        (b') to (a')
        (b') to (c')
        (c') to (a');
        \node[rectangle,draw=none,fill=none] at (0,-2) {co-domino};
      \end{scope}
      \begin{scope}[xshift=6cm]
        \node[label=right:{$\strut v_1$}] (a) at (90:1) {};
        \node[label=right:{$\strut a_1$}] (a') at (90:2) {};
        \node[label=below right:{$\strut v_3$}] (b) at (210:1) {};
        \node[label=left:{$\strut a_3$}] (b') at (210:2) {};
        \node[label=below left:{$\strut v_2$}] (c) at (-30:1) {};
        \node[label=right:{$\strut a_2$}] (c') at (-30:2) {};

        \draw (a) to (b)
        (b) to (c)
        (c) to (a)
        (a) to (a')
        (b) to (b')
        (c) to (c');
        \node[rectangle,draw=none,fill=none] at (0,-1.5) {net};
      \end{scope}
    \end{tikzpicture}
    \caption{The co-domino and the net.}\label{co-domino and net figure}
  \end{figure}
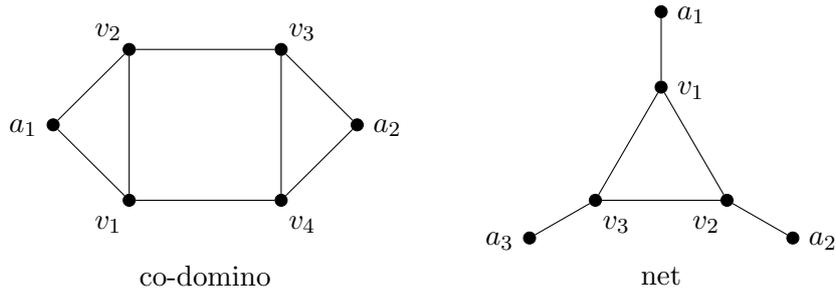

  Consider the following algorithm whose input is a pyramid-free graph $G$.
  \begin{enumerate}
    \item\label[step]{enum co-domino} Enumerate all the induced subgraphs of $G$ isomorphic to a
          co-domino (with notation as in the definition).
          For each of them compute the connected components of $G\setminus (N(v_1) \cup N(v_2) \cup N(v_3) \cup N(v_4))$.
          If one of them contains neighbors of both $a_1$ and $a_2$, output that some long prism exists in $G$ and stop the algorithm.

    \item\label[step]{enum nets} Enumerate all induced subgraphs of $G$ isomorphic to a net (with
          notation as in the definition). For each of them compute the
          connected components of
          $G\setminus (N(v_1) \cup N(v_2) \cup N(v_3))$.  If one of them
          contains neighbors of $a_1$, $a_2$, and $a_3$, output that some long
          prism exists in $G$ and stop the algorithm.

    \item\label[step]{output no long prism} Output that $G$ contains no long prism.
  \end{enumerate}

  The algorithm relies on brute force enumeration of sets
  on six vertices and computation of connected components, so it can
  be implemented to run in time $\mathcal{O}(n^6(n+m))$.
  Also, it clearly outputs that $G$ contains a long prism or that it does not.
  It therefore remains to prove that the algorithm indeed gives the answer that $G$ contains a long prism if and only if it does.

  If $G$ contains a long prism $H$, then assume first that two of the paths of $H$ have length~1.
  Then $H$ contains an induced co-domino that is detected in the \cref{enum co-domino} of the algorithm.
  Because of the third path of the prism, at least one of the computed connected components contains neighbors of both $a_1$ and $a_2$, so the algorithm gives the correct answer.
  Assume now that at least two paths of $H$ have length at least~2.
  Then $H$ contains a net that is detected in \cref{enum nets} of the algorithm (unless the algorithm stopped before, which is fine since $G$ contains a long prism by assumption).
  At least one of the computed connected components (the one that contains the rest of the prism) contains neighbors of $a_1$, $a_2$, and $a_3$, so the algorithm gives the correct answer and stops.

  Conversely, assume that the algorithm gives the answer that $G$ contains a long prism.

  If such an answer is given in \cref{enum co-domino}, this is because of some co-domino and some connected component $X$ that contains neighbors of $a_1$ and $a_2$.
  A shortest path from $a_1$ to $a_2$ with interior in $X$ together with the co-domino form a long prism, so a long prism does exist in the graph.

  If such an answer is given in \cref{enum nets}, it is because of some connected component $X$ of $G\setminus (N(v_1) \cup N(v_2) \cup N(v_3))$ that contains neighbors of $a_1$, $a_2$, and $a_3$.
  In $G' = G[X \cup \{a_1, a_2, a_3\}]$, let us apply \cref{l:con3} to obtain an induced subgraph $H$ that belongs to $\mathcal{S} \cup \mathcal{T} \cup \mathcal{M}$ and has $\{a_1, a_2, a_3\}$ as the set of extremities.
  If $H \in \mathcal{S}$ then $G$ contains a pyramid with triangle $\{v_1,v_2,v_3\}$; this cannot happen since $G$ is pyramid-free.
  If $H \in \mathcal{T}$, then $G$ contains a long prism and the output of the algorithm is correct.
  If $H \in \mathcal{M}$ then $G$ contains a pyramid if the center of $H$ has two nonadjacent neighbors in $H$ or a long prism otherwise.
  The first case cannot happen by assumption and hence the output of the algorithm is correct.
\end{proof}

\subsection{Detecting a broken wheel}
\label{sec:bw}

We give here a polynomial-time algorithm to detect a broken wheel in a graph with no long prism, no pyramid, and no theta.
Recall that detecting a broken wheel in general is an \NP-complete problem.
Note that this article addresses the question of the detection of wheels (broken or not), however, the proof (of~\cite[Theorem 4]{MR3289471}) that detecting a wheel is \textsf{NP}-complete proves without a single modification that detecting a broken wheel is \textsf{NP}-complete.

Let $W = (H, x)$ be a broken wheel.
Since $W$ is broken, it contains at least two sectors of length at least~2, say $P= a\dots b$ and $R = c\dots d$.
We orient $H$ clockwise and assume that $a, b, c$, and $d$ appear clockwise in this order along $H$.
We denote by $Q$ the path of $H$ from $b$ to $c$ that does not contain $a$ and $d$.
We denote by $S$ the path of $H$ from $d$ to $a$ that does not contain $b$ and $c$.
Note that
\begin{equation}\label{eq:wheel}
  \text{possibly $b=c$ or $d=a$, but not both,}
\end{equation}
since $W$ is a wheel by assumption.
Note also that $a\neq b$, $c\neq d$, $ab\notin E(G)$, and $cd\notin E(G)$.

For every vertex $v$ in $H$, we denote by $v^+$ its neighbor in $H$ in the clockwise direction, and by $v^-$ its neighbor in the counter-clockwise direction.
Note that possibly $a^+ = b^-$, $a^+b^-\in E(G)$, $c^+ = d^-$, or $c^+d^-\in E(G)$.

We say that the 13-tuple $F = (x, a, b, c, d, a^+, b^+, c^+, d^+, a^-, b^-, c^-, d^-)$ is a \emph{frame} for $W$.
Every broken wheel has at least one frame.
In what follows, once a frame is fixed, we use the notation for the paths $P$, $Q$, $R$, $S$ without recalling it; see \cref{f.frame}.

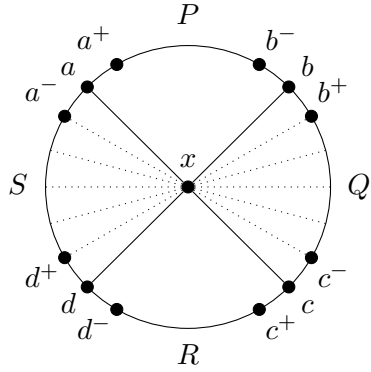
\begin{figure}
  \centering
  \begin{tikzpicture}[scale=0.75]
    \tikzset{every node/.style={draw,circle,fill=black,inner sep=0pt,minimum size=4.5pt}}
    \draw circle (2.5);
    \node[label=above:{$\strut x$}] (x) at (0,0) {};

    \node[label=above left:{$\strut a^+$}] (a+) at (120:2.5) {};
    \node[label=above left:{$\strut a$}] (a) at (135:2.5) {};
    \node[label=above left:{$\strut a^-$}] (a-) at (150:2.5) {};

    \node[label=above right:{$\strut b^+$}] (b+) at (30:2.5) {};
    \node[label=above right:{$\strut b$}] (b) at (45:2.5) {};
    \node[label=above right:{$\strut b^-$}] (b-) at (60:2.5) {};

    \node[label=below right:{$\strut c^-$}] (c-) at (-30:2.5) {};
    \node[label=below right:{$\strut c$}] (c) at (-45:2.5) {};
    \node[label=below right:{$\strut c^+$}] (c+) at (-60:2.5) {};

    \node[label=below left:{$\strut d^-$}] (d-) at (-120:2.5) {};
    \node[label=below left:{$\strut d$}] (d) at (-135:2.5) {};
    \node[label=below left:{$\strut d^+$}] (d+) at (-150:2.5) {};

    \node[draw=none,fill=none,rectangle,label=left:$\strut S$] at (180:2.5) {};
    \node[draw=none,fill=none,rectangle,label=above:$\strut P$] at (90:2.5) {};
    \node[draw=none,fill=none,rectangle,label=right:$\strut Q$] at (0:2.5) {};
    \node[draw=none,fill=none,rectangle,label=below:$\strut R$] at (-90:2.5) {};

    \draw (x) to (a)
    (x) to (b)
    (x) to (c)
    (x) to (d);

    \foreach \t in {150,165,180,195,210,30,15,0,-15,-30}{
        \draw[dotted] (x) to (\t:2.5);
      }
  \end{tikzpicture}
  \caption{A frame for a broken wheel. Dotted edges may be present or not.}\label{f.frame}
\end{figure}

If $W$ is a broken wheel in a graph $G$ with frame $F$, we say that
$(W, F)$ is \emph{optimal} if:
\begin{itemize}
  \item no broken wheel of $G$ has fewer vertices than $W$, and
  \item among all broken wheels with the same number of vertices as $W$
        and their frames, the sum of the lengths of $S$ and $Q$ is minimal.
\end{itemize}

The \emph{cube} is the graph with vertex set $\{v_1, v_2, v_3, v_4, v_5, v_6, x, y\}$ such that $v_1v_2v_3v_4v_5v_6v_1$ is a hole, $N(x) = \{v_1, v_3, v_5\}$, and $N(y) = \{v_2, v_4, v_6\}$; see \cref{fig:cube}.
Observe that removing any vertex from the cube yields a broken wheel on seven vertices.
Hence, every graph that does not contain a broken wheel is cube-free.

\begin{figure}[tb]
  \centering
  \begin{tikzpicture}[scale=2]
    \tikzset{every node/.style={draw,circle,fill=black,inner sep=0pt,minimum size=4.5pt}}
    \foreach \x/\a in {a/90,b/30,c/-30,d/-90,e/210,f/150}{
        \node (\x) at (\a:1) {};
      }
    \draw (a) to (b) to (c) to (d) to (e) to (f) to (a);
    \node (x) at (210:.4) {};
    \draw (x) to[bend left] (b)
    (x) to (d)
    (x) to (f);
    \node (y) at (30:.4) {};
    \draw (y) to (a)
    (y) to (c)
    (y) to[bend left] (e);
  \end{tikzpicture}
  \caption{The cube.}
  \label{fig:cube}
\end{figure}
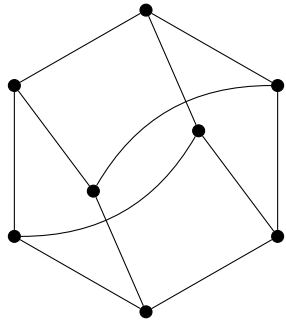

Given a graph $G$, a wheel $W = (H, x)$ contained in $G$ as an induced subgraph, and a vertex $y\in V(G)\setminus V(W)$, we denote by $N_H(y)$ the set of vertices in $H$ that are adjacent to $y$.

We start with three lemmas analyzing the structure around an optimal pair $(W, F)$ of a broken wheel and a frame.

\begin{lemma}\label{l:nw}
  Let $G$ be a (long prism, pyramid, theta, cube)-free graph.
  Let $W = (H, x)$ be a broken wheel in $G$ with frame $F = (x, a, b, c, d, a^+, b^+, c^+, d^+$, $a^-, b^-,\allowbreak c^-, d^-)$ such that $(W, F)$ is optimal.
  Let $y$ be a vertex of $G\setminus V(W)$ that is not adjacent to $a$, $b$, $c$, or $d$.
  Then $N_H(y)$ induces a (possibly empty) path of length at most~2.
  In particular, $N_H(y)$ is contained in the interior of one of $P$, $Q$, $R$, or $S$.
\end{lemma}

\begin{proof}
  Assume first that $y$ has neighbors in at most one path among $P$, $Q$, $R$, and $S$.
  If $N_H(y)$ is not contained in some path of $H$ of length at most~2, it can be used to replace some subpath of $H$ to obtain a broken wheel with fewer vertices than $W$, a contradiction (note that this is correct if $xy\in E(G)$ and if $xy\notin E(G)$).
  Hence, $N_H(y)$ is contained in some path $T$ of $H$ of length at most~2, and it must be equal to $T$, because otherwise $T$ has length~2 and $y$ is nonadjacent to its middle vertex, and $y$ and $H$ form a theta.
  So, the conclusion of the lemma holds.
  Hence, from now on, we assume that $y$ has neighbors in at least two paths among $P$, $Q$, $R$, and $S$.

  If $N_H(y) \subseteq Q\cup S$, then $y$ has neighbors in both $Q$ and $S$.
  Since $G$ is theta-free, $y$ has at least three neighbors
  in $H$.
  It follows that $y$ is the center of a broken wheel with
  rim $H$, with the same number of vertices as $W$ and that has a
  frame contradicting the minimality of the sum of the lengths of $Q$
  and $S$ (since $y$ is not adjacent to the endpoints of $Q$ and $S$).
  Hence, $y$ has neighbors in the interior of at least one of $P$ or $R$, say in $P$ up to symmetry.

  If $y$ is adjacent to $x$, then $P$, $x$, and $y$ induce a theta or a broken wheel with center $y$ and with fewer vertices than $W$, a contradiction.
  So, $y$ is not adjacent to $x$.

  Assume that $y$ has neighbors in $S$.
  Let $a'$ be the neighbor of $y$ in $P$ closest to $a$ along $P$.
  Let $b'$ be the neighbor of $y$ in $P$ closest to $b$ along $P$.
  Let $d'$ be the neighbor of $y$ in $S$ closest to $d$ along $S$.
  Let $T$ be a shortest path from $x$ to $d'$ in the subgraph of $G$ induced by the vertices of the path $xdSd'$.
  If $d'a\notin E(G)$, the paths $xaPa'$, $xbPb'$, and $xTd'y$ induce a theta or a pyramid (namely, if $a'=b'$ they form a theta with apices $a'$ and $x$, if $a'b'\in E(G)$ they form a pyramid, and otherwise they form a theta with apices $x$ and $y$).
  Hence, $d'a\in E(G)$.
  It follows that $d'=a^-$ is the unique neighbor of $y$ in $S$.

  We claim that $a'=b'$ and $a'a\in E(G)$.
  Indeed, if $a'=b'$ and $a'a\not\in E(G)$, then the three paths $aPa'$, $axbPa'$, and $ad'ya'$ form a theta with apices $a$ and $a'$; if $a'$ and $b'$ are adjacent, then the three paths $aPa'$, $axbPb'$, and $ad'y$ form a pyramid with apex $a$ and triangle $\{a',b',y\}$; and if $a'$ and $b'$ are distinct and nonadjacent, then the three paths $aPa'y$, $axbPb'y$, and $ad'y$ form a theta with apices $a$ and $y$.
  It follows that $a'=a^+$ is the unique neighbor of $y$ in $P$.
  Therefore, the hole $xTa^-ya^+Pbx$ is the rim of a broken wheel $W'$
  centered at $a$.
  To avoid that the wheel $W'$ has fewer vertices than $W$, we infer that $b=c$, $R$ has length exactly~2, and $T = xdSa^-$ (which implies $xa^-\notin E(G)$).
  So, $Q= b$ and $R=bb^+d$.
  Vertex $y$ must be adjacent to $b^+$, since otherwise $H$ and $y$ form a theta with apices $a^-$ and $a^+$ (recall that $y$ is not adjacent to $a,b,c$ and $d$).
  We have $a^-d\in E(G)$, for otherwise the three paths $dSa^-$, $dxaa^-$, and $db^+ya^-$ form a theta with apices $d$ and $a^-$.
  By a symmetric argument, we must have $a^+b\in E(G)$.
  Hence, $\{x, y, a, b, d, a^-, a^+, b^+\}$ induces a cube, a contradiction.

  We proved that $y$ cannot have neighbors in $S$. Symmetrically, it
  cannot have neighbors in $Q$. Hence, $y$ has neighbors in both $P$
  and $R$ and no neighbor in $Q$ and $S$.  Since $G$ contains no
  theta, $y$ must have at least two neighbors in either $P$ or $R$,
  say in $R$ up to symmetry.
  Up to symmetry, we may assume by \eqref{eq:wheel} that
  $a\neq d$ (so possibly $b=c$).
  Let $a'$ be the neighbor of $y$ in $P$ closest to $a$ along $P$.
  Let $d'$ be the neighbor of $y$ in $R$ closest to $d$ along $R$.
  Let $c'$ be the neighbor of $y$ in $R$ closest to $c$ along $R$.
  Note that $c'\neq d'$.  If
  $ca'\notin E(G)$, then the three paths $yd'Rdx$, $yc'Rcx$, and $ya'Pax$ induce a long prism, a pyramid, or a theta (namely, they induce a long prism if $c'd'\in E(G)$ and $ad\in E(G)$, a pyramid if exactly one of $c'd'\in E(G)$ and $ad\in E(G)$ happens, and a theta otherwise).  So, $ca'\in E(G)$.
  Hence, $b=c$ and $a'$ is the unique neighbor of $y$ in $P$.
  Hence, since $c'\neq d'$, the three paths $yd'Rdxb$, $yc'Rb$, and $ya'b$ induce a pyramid with triangle $\{y,c',d'\}$ or a theta with apices $y$ and $b$, a contradiction.
\end{proof}

\begin{lemma}\label{l:rP}
  Let $G$ be a (long prism, pyramid, theta, cube)-free graph.
  Let $W = (H, x)$ be a broken wheel of $G$ with frame $F = (x, a, b, c, d, a^+, b^+, c^+, d^+$, $a^-, b^-,\allowbreak c^-, d^-)$ such that $(W, F)$ is optimal.
  Let $P' = a^+ \dots b^-$ be a shortest path from $a^+$ to $b^-$ in  $G\setminus( (N[a]\cup N[b] \cup N[c] \cup N[d] \cup N[x])
    \setminus\{a^+, b^-\})$.
  Then the induced subgraph of $G$ obtained from $W$ by removing the internal vertices of $P$ and adding the vertices of $P'$ instead is a broken wheel with frame $F$, and $(W', F)$ is optimal.
  A similar result holds for $R$.
\end{lemma}
\begin{proof}
  If no vertex of $P'$ has a neighbor in the interior of $Q$, $R$, or $S$ (in which case $P'$ is vertex-disjoint from $Q\cup R \cup S$), then the conclusion of the lemma holds.
  Hence, suppose for a contradiction that some vertex $v$ of $P'$ has a neighbor in the interior of $Q$, $R$, or $S$, and choose $v$ closest to $a^+$ along $P'$.
  Note that $v$ is in the interior of $P'$, hence, $v\not\in W$ and $v$ is not adjacent to any of $a$, $b$, $c$, $d$, and $x$.
  Let $w$ be the vertex closest to $v$ along $vP'a^+$ that has some neighbor in $P$.
  Note that $w$ exists because $a^+\in P$; moreover, $w$ is in the interior of $P'$; hence, $w\not\in W$ and $w$ is not adjacent to any of $a$, $b$, $c$, $d$, and $x$.
  Since $v$ has a neighbor in the interior of $Q$, $R$, or $S$, \cref{l:nw} implies that $N_H(v)$ induces a nonempty path of length at most~2 that is contained in the interior of either $Q$, $R$, or $S$.
  Similarly, $N_H(w)$ induces a nonempty path of length at most~2 that is contained in the interior of $P$.
  Note that $w\neq v$.
  We set $T= vP'w$.
  Now, we focus on $T$ and $H$.
  By construction, $T$ is disjoint from $H$ and no vertex in the interior of $T$ has a neighbor in $H$.
  If $vw \notin E(G)$ or if one of $\vert N_H(v) \vert$ or $\vert N_H(w) \vert$ has size at most $2$, then $T\cup H$ induces a long prism, a pyramid, or a theta.
  Hence, $vw\in E(G)$, and both $N_H(v)$ and $N_H(w)$ induce a path of length~2.
  It follows that the subgraph of $G$ induced by $H\cup T \cup  \{x\}$ contains a theta with apices $v$ and $x$, a contradiction.
\end{proof}

The next lemma might look identical to the previous one, but there is
an important difference: the vertices of the shortest path are allowed
to be adjacent to $x$.

\begin{lemma}
  \label{l:rQ}
  Let $G$ be a (long prism, pyramid, theta, cube)-free graph.  Let
  $W = (H, x)$ be a broken wheel of $G$ with frame
  $F = (x, a, b, c, d, a^+, b^+, c^+, d^+, a^-,$ $b^-, c^-, d^-)$ such
  that $(W, F)$ is optimal.
  Let $Q' = b^+ \dots c^-$ be a shortest path from $b^+$ to $c^-$ in
  $G\setminus( (N[a]\cup N[b] \cup N[c] \cup N[d])
    \setminus\{b^+, c^-\})$.
  Then the induced subgraph of $G$ obtained from $W$ by removing the internal vertices of $Q$ and adding the vertices of $Q'$ instead is a broken wheel with frame $F$, and  $(W', F)$ is optimal.
  A similar result holds for $S$.
\end{lemma}

\begin{proof}
  If no vertex of $Q'$ has a neighbor in the interior of $P$, $R$, or $S$ (in which case $Q'$ is vertex-disjoint from $P\cup R \cup S$), then the conclusion of the lemma holds.
  Hence, suppose for a contradiction that some vertex $v$ of $Q'$ has some neighbor in the interior of $P$, $R$, or $S$, and choose $v$ closest to $b^+$ along $Q'$.
  Let $w$ be the vertex closest to $v$ along $vQ'b^+$ that has some neighbor in $Q$.

  The rest of the proof is similar to the proof of \cref{l:rP}, except that $x$ might be adjacent to $v$ or $w$.
  For clarity, we provide in this paragraph the details of the proof that are similar to those in the proof of \cref{l:rP}.
  Note that $v$ is in the interior of $Q'$, hence, $v\not\in W$ and $v$ is not adjacent to any of $a$, $b$, $c$, and $d$.
  Similarly, $w$ exists because $b^+\in Q$; moreover, $w$ is in the interior of $Q'$; hence, $w\not\in W$ and $w$ is not adjacent to any of $a$, $b$, $c$, and $d$.
  Since $v$ has a neighbor in the interior of $P$, $R$, or $S$, \cref{l:nw} implies that $N_H(v)$ induces a nonempty path of length at most~2 that is contained in the interior of either $P$, $R$, or $S$.
  Similarly, $N_H(w)$ induces a nonempty path of length at most~2 that is contained in the interior of $Q$.
  Since $v$ has no neighbors in $Q$, it is distinct from $w$ and nonadjacent to both $b^+$ and $c^-$.
  We set $T= vQ'w$.
  Now we focus on $T$ and $H$.
  By construction, $T$ is disjoint from $H$ and no vertex in the interior of $T$ has a neighbor in $H$.
  If $vw \notin E(G)$ or if one of $\vert N_H(v) \vert$ or $\vert N_H(w) \vert$ has size at most $2$, then $T\cup H$ induces a long prism, a pyramid, or a theta.
  Hence, $vw\in E(G)$, and both $N_H(v)$ and $N_H(w)$ induce a path of length~2.
  If $xv\notin E(G)$, then the subgraph of $G$ induced by $H\cup T \cup  \{x\}$ contains a theta with apices $x$ and $v$.
  If $xw\notin E(G)$, then the subgraph of $G$ induced by $H\cup T \cup \{x\}$ contains a theta with apices $x$ and $w$.
  Hence, $xv\in E(G)$ and $xw\in E(G)$.

  Recall that $N_H(v)$ induces a path of length~2 that is contained in the interior of exactly one of $P$, $R$, or $S$.
  We now analyze these three cases, starting with $S$.
  If $N_H(v)\subseteq S$, then let $w'$ be the middle vertex of $N_Q(w)$.
  The hole $H'$ obtained from $H$ by replacing $w'$ by $w$ is the rim
  of a broken wheel $W'$ centered at $v$.
  This wheel has a frame that contradicts the optimality of $(W, F)$.

  If $N_H(v)\subseteq P$, then some hole going through $a$, $v$, $w$,
  $c$, and $d$ is the rim of a broken wheel centered at $x$ that has
  fewer vertices than $W$, a contradiction to the optimality of
  $(W, F)$.
  The case if $N_H(v)\subseteq R$ is similar.
\end{proof}

Now we have everything ready to prove the main result of this section.

\begin{lemma}
  \label{l:bW}
  Determining whether a given (long prism, pyramid, theta)-free graph contains a broken wheel as an induced subgraph can be done in time   $\mathcal{O}(n^{13}(n+m))$.
\end{lemma}

\begin{proof}
  Consider the following algorithm whose input is any (long prism, pyramid, theta)-free graph~$G$.
  \begin{enumerate}
    \item\label[step]{enumerate 7-tuples} Enumerate all 7-tuples of vertices of $G$, and if one of them induces a broken wheel, output ``$G$ contains a broken wheel'' and stop.
    \item\label[step]{enumerate 13-tuples} Enumerate all 13-tuples
          $F= (x$, $a,$ $b,$ $c,$ $d,$ $a^+,$ $b^+,$ $c^+,$ $d^+,$ $a^-,$ $b^-,$ $c^-,$ $d^-)$ of vertices of $G$.
          For each of them:

          \begin{enumerate}[label=(\roman*)]
            \item\label[substep]{a+ to b-} Compute a shortest path $P'$ from $a^+$ to $b^-$ in
                  $G\setminus( (N[a]\cup N[b] \cup N[c] \cup N[d] \cup N[x])
                    \setminus\{a^+, b^-\})$.
                  If no such path exists, set $P' = \emptyset$.
            \item\label[substep]{b+ to c-} If $b = c$, set $Q' = \emptyset$. Otherwise, compute a shortest path $Q'$ from $b^+$ to $c^-$ in $G\setminus( (N[a]\cup N[b] \cup N[c] \cup N[d])
                    \setminus\{b^+, c^-\})$. If no such path exists, set $Q' = \emptyset$.
            \item\label[substep]{c+ to d-} Compute a shortest path $R'$ from $c^+$ to $d^-$ in $G\setminus( (N[a]\cup N[b] \cup N[c] \cup N[d] \cup N[x])
                    \setminus\{c^+, d^-\})$.
                  If no such path exists, set $R' = \emptyset$.
            \item\label[substep]{d+ to a-} If $d = a$, set $S' = \emptyset$. Otherwise, compute a shortest path $S'$ from $d^+$ to $a^-$ in $G\setminus( (N[a]\cup N[b] \cup N[c] \cup N[d])
                    \setminus\{d^+, a^-\})$.
                  If no such path exists, set $S' = \emptyset$.
            \item\label[substep]{output contains broken wheel} If $P'\cup Q' \cup R'\cup S'\cup F$ induces a broken wheel centered at $x$, output
                  ``$G$ contains a broken wheel'' and stop.
          \end{enumerate}
    \item Output ``$G$ contains no broken wheel''.
  \end{enumerate}

  If the algorithm outputs that $G$ contains a broken wheel in \cref{enumerate 13-tuples}~\cref{output contains broken wheel}, it obviously does.
  Conversely, assume that $G$ contains a broken wheel, and let us check that  the algorithm gives the correct answer.

  If $G$ contains a broken wheel on~7 vertices, then it is detected in \cref{enumerate 7-tuples}.
  So we may assume that $G$ contains no broken wheel on~7 vertices, which implies that $G$ contains no cube and that the algorithm continues with \cref{enumerate 13-tuples}.

  Since $G$ contains a broken wheel, it contains a broken wheel $W$ with a frame $F$ such that $(W, F)$ is optimal.
  We denote by $P$, $Q$, $R$, and $S$ the paths of the broken wheel as above.
  At some point in \cref{enumerate 13-tuples}, the algorithm considers $F$.
  By \cref{l:rP} applied to $W$, the paths $P'$ (defined in \cref{enumerate 13-tuples}~\cref{a+ to b-}), $Q$, $R$, and $S$ together with $F$ form a broken wheel $W_1$.
  By \cref{l:rQ} applied to $W_1$ if $b\neq c$ (and trivially otherwise), the paths $P'$, $Q'$ (defined in \cref{enumerate 13-tuples}~\cref{b+ to c-}), $R$, and $S$ together with $F$ form a broken wheel $W_2$.
  By \cref{l:rP} applied a second time to $W_2$, the paths $P'$, $Q'$, $R'$ (defined in \cref{enumerate 13-tuples}~\cref{c+ to d-}), and $S$ together with $F$ form a broken wheel $W_3$.
  By \cref{l:rQ} applied a second time to $W_3$ if $d\neq a$ (and trivially otherwise), the paths $P'$, $Q'$, $R'$, and $S'$ (defined in \cref{enumerate 13-tuples}~\cref{d+ to a-}) together with $F$ form a broken wheel $W_4$.
  Hence, the algorithm actually finds a broken wheel, and therefore gives the correct answer.

  The algorithm relies on a brute force enumeration of 13-tuples of vertices that can be implemented in time $\mathcal{O}(n^{13})$, and a search of the graph for computing shortest paths $P'$, $Q'$, $R'$, $S'$, as well as testing if $P'$, $Q'$, $R'$, $S'$, and $F$ induce a broken wheel centered at $x$ can be implemented in time $\mathcal{O}(n+m)$.
\end{proof}

\subsection{Deciding if $K_{2,3}$ is an induced minor}
\label{sec:o}

We can now state the main result of this section.

\begin{theorem}
  Determining whether a given graph contains $K_{2, 3}$ as an induced minor can be done in time $\mathcal{O}(n^{13}(n+m))$.
\end{theorem}

\begin{proof}
  By \cref{t:k23T}, it is enough to decide whether $G$ contains a long prism, a pyramid, a theta, or a broken wheel as an induced subgraph.
  Detecting a pyramid can be performed in time $\widetilde{O}(n^5)$ by \cref{th:pyramid}.
  Detecting a theta can be performed in time $\widetilde{O}(n^6)$ by \cref{th:theta}.
  Hence, it can be assumed that the input graph contains no pyramid and no theta, so that detecting a long prism can be performed in time $\mathcal{O}(n^6(n+m))$ by \cref{l:longP}.
  Hence, it can be assumed that the input graph contains no pyramid, no theta, and no long prism, so that detecting a broken wheel can be performed in time $\mathcal{O}(n^{13}(n+m))$ by \cref{l:bW}.
\end{proof}

All the algorithms in this section are written for the decision version of the problem they address, but an algorithm that actually outputs the structure whose existence is guaranteed can be obtained using \cref{obs sub algo,lem algo}.

The above polynomial-time algorithm suggests that a full structural description of graphs that do not contain $K_{2, 3}$ as an induced minor may be worth investigating.
In particular, we propose the following.

\begin{conjecture}\label{conj}
  There exists a polynomial $p$ and an integer $k$ such that if $G$ has no clique cutset and does not contain $K_{2,3}$ as an induced minor, then either $G$ has at most $p(|V(G)|)$ minimal separators or $G$ has no hole of length at least $k$.
\end{conjecture}

If \cref{conj} is true, it would give an alternative polynomial-time recognition algorithm for determining whether a given graph contains $K_{2, 3}$ as an induced minor.
Indeed, if a graph $G$ has a polynomial number of minimal separators, then they can all be listed in polynomial time~\cite{MR1792122}.
On the other hand, if $G$ does not contain cycles of length at least $k$, then the list of induced subgraphs from \cref{t:k23T} becomes finite, and thus detecting their presence can be done with a brute-force approach in polynomial-time.

It is not difficult to construct examples showing that the condition that $G$ does not contain any hole of length at least $k$ cannot be replaced with the condition that $G$ has independence number at most $k$.
For an integer $k \geq 2$, let $G_k$ be the graph with $V(G_k) = A\cup B\cup C\cup D$ where $A$, $B$, $C$, and $D$ are pairwise disjoint sets with $k$ vertices each, $A = \{a_1,\ldots, a_k\}$ is an independent set, each of $B = \{b_1,\ldots, b_k\}$, $C = \{c_1,\ldots, c_k\}$, and $D = \{d_1,\ldots, d_k\}$ is a clique, every vertex in $A$ is adjacent to every vertex of $C$, every vertex in $B$ is adjacent to every vertex of $D$, for every $i,j\in \{1,\ldots, k\}$, vertices $a_i$ and $b_j$ are adjacent if and only if $i = j$, vertices $b_i$ and $c_j$ are adjacent if and only if $i \neq j$, vertices $c_i$ and $d_j$ are adjacent if and only if $i = j$, and there are no other edges.
Then, $A\cup \{d_1\}$ is an independent set in $G_k$ with cardinality $k+1$, and it can be verified that $G_k$ has no clique cutset, does not contain $K_{2,3}$ as an induced minor (this can be checked using \cref{t:k23T} or by observing that no minimal separator includes an independent of size $3$), and the number of minimal separators of graphs in the family $\{G_k\colon k\ge 2\}$ is not bounded by any polynomial function in the number of vertices (the graph $G[C \cup D]$, sometimes referred to as the (short) \emph{$k$-prism}, has at least $2^{k-2}$ number of minimal separators).

\section*{Acknowledgements}

  We are grateful to the anonymous reviewers for their helpful remarks.
Special thanks to Sebastian Wiederrecht for insightful discussions and for suggesting the idea that led us to \cref{lem algo}.

\printbibliography

\end{document}